\documentclass[12pt, a4paper]{amsart}

\usepackage[utf8]{inputenc}

\usepackage{amsmath}
\usepackage{amssymb}
\usepackage{amsfonts}
\usepackage{amsthm}

\usepackage{tikz-cd}
\usepackage{tikz}
\usetikzlibrary{patterns,shapes,arrows}
\usepackage{array}
\usepackage{setspace}
\usepackage{url}
\usepackage{enumerate}
\usepackage{comment}
\usepackage{graphicx}
\usepackage{hyperref}
\hypersetup{colorlinks=true,linkcolor=black,citecolor=black,urlcolor=black}
\usepackage{caption}
\usepackage{subcaption}

\swapnumbers

\newtheoremstyle{numpar}
{6pt}			
{6pt}			
{}			
{}			
{}			
{\ }			
{ }			
{}			
\theoremstyle{numpar}

\newtheorem{num}{\unskip}[subsection]

\theoremstyle{plain}
\newtheorem{theorem}[num]{Theorem}
\newtheorem{lemma}[num]{Lemma}
\newtheorem{prop}[num]{Proposition}

\newtheorem{corollary}[num]{Corollary}
\newtheorem{demo*}[num]{Proof}

\theoremstyle{definition}
\newtheorem{definition}[num]{Definition}
\newtheorem{example}[num]{Example}

\theoremstyle{remark}
\newtheorem{remark}[num]{Remark}

\numberwithin{equation}{subsection}

\DeclareFontFamily{U}{wncy}{}
\DeclareFontShape{U}{wncy}{m}{n}{<->wncyr10}{}
\DeclareSymbolFont{mcy}{U}{wncy}{m}{n}
\DeclareMathSymbol{\Sha}{\mathord}{mcy}{"58}

\newcommand{\ints}{\mathbb{Z}}

\newcommand{\rationals}{\mathbb{Q}}
\newcommand{\reals}{\mathbb{R}}
\newcommand{\complex}{\mathbb{C}}

\newcommand{\Frob}{\operatorname{Frob}}
\newcommand{\Hom}{\operatorname{Hom}}
\newcommand{\Aut}{\operatorname{Aut}}
\newcommand{\End}{\operatorname{End}}

\newcommand{\id}{\operatorname{id}}
\newcommand{\Supp}{\operatorname{Supp}}
\newcommand{\Gal}{\operatorname{Gal}}

\newcommand{\Res}{\operatorname{Res}}

\newcommand{\Spec}{\operatorname{Spec}}
\newcommand{\Proj}{\operatorname{Proj}}

\newcommand{\GL}{\operatorname{GL}}
\newcommand{\SL}{\operatorname{SL}}
\newcommand{\PSL}{\operatorname{PSL}}

\newcommand{\Sp}{\operatorname{Sp}}

\newcommand{\Sym}{\operatorname{Sym}}

\newcommand{\True}{\operatorname{True}}
\newcommand{\False}{\operatorname{False}}

\calclayout

\title{Dwork Motives, Monodromy and Potential Automorphy}
\author{Lambert A'Campo}
\date{\today}

\begin{document}

\begin{abstract}
	In this paper we study certain families of motives, which arise as direct summands of the cohomology of the Dwork family. We computationally find examples of interesting families with the following three properties. Firstly, their geometric monodromy group is Zariski dense in $\SL_n$. Secondly, they realise many different unipotent operators as the monodromy operator at $t = \infty$. Thirdly, all their Hodge numbers are $\leq 1$.
	
	This has consequences for Galois representations. Namely, if a nilpotent operator $N$ appears as the monodromy at $t = \infty$ in one of our families, we can construct potentially automorphic representations with $\ell$-adic monodromy given by $N$ at a fixed prime $p$.
	As another application, we obtain a new proof of some cases of the recent local-global compatibility theorem of Matsumoto.
\end{abstract}
	
	\maketitle
	
	\tableofcontents
	
	\section{Introduction}
	
	Let $d \geq 3$ be an integer and consider the Dwork family of projective hypersurfaces
	\[
		Y_t : x_0^d + x_1^d + \dots + x_{d-1}^d = d t x_0 x_1 \cdots x_{d-1}.
	\]
	In this paper we study certain families of motives $M_t$, which are $\ints[\zeta_d]$-linear direct summands of $H^{d-2}(Y_t, \ints[\frac{1}{d}, \zeta_d])$, where $\zeta_d$ is a primitive $d$th root of unity. Let $N$ be a $n \times n$ nilpotent matrix. For $n \leq 5$ and all possible choices of $N$, we have computationally found a value of $d$ and a direct summand $M_t$ of dimension $n$, such that the following special properties are satisfied (see Proposition \ref{computed_motives_prop} for a precise statement)
	\begin{enumerate}
		\item[(BM)] Big monodromy: The Zariski closure of the geometric monodromy group $G_{geom} \subset \GL(M_{B,t} \otimes \complex)$ contains $\SL_n$, where $M_{B,t}$ is the Betti realisation of $M_t$.
		\item[(R)] Regular Hodge--Tate weights: The Hodge numbers of $M_{dR,t} \otimes_{\iota} \complex$ are $\leq 1$ for every embedding $\iota : \rationals(\zeta_d) \to \complex$, where  $M_{dR,t}$ is the de Rham realisation of $M_t$.
		\item[(UM)] Unipotent monodromy at infinity: The monodromy operator at $t = \infty$ acting on $M_{B,t}$ is conjugate to $\exp(N)$.
	\end{enumerate}
	
	When $N$ has maximal rank, such families are known for all $n$ by the work of Barnet-Lamb \cite{BL_even}, \cite{BL_odd}. For general $N$ it turns out to be (at least computationally) more difficult to find direct summands which simultaneously satisfy (BM), (R) and (UM). We have not succeeded in finding a general recipe.
	However, when a family $M_t$ with these properties exists for a given $N$, we can apply the method of \cite{HSBT} to $M_t$ to prove a potential automorphy theorem with additional control on local monodromy. This theorem and its consequences will be introduced in \S \ref{galois_intro} below.

	\subsection{Families of Motives} 
	Let $0 < n < d$ be an integer. We define a hypergeometric parameter of dimension $n$ modulo $d$ as a tuple 
 	\[
 		(\alpha_1, \dots, \alpha_n; \beta_1, \dots, \beta_n)
 	\] 
 	of elements of  $\ints/d\ints$ such that
\begin{itemize}
	\item $\sum_i \alpha_i - \sum_i \beta_i = \binom{d}{2}$,
	\item $\alpha_i \neq \beta_j$ for $i, j \in \{1, \dots, n\}$,
	\item $\beta_i \neq \beta_j$ for $i \neq j$.
\end{itemize}
Associated with such a parameter, one obtains an $n$-dimensional family of motives 
\[
M(\alpha_1, \dots, \alpha_n ; \beta_1, \dots, \beta_n)_t
\]
which is a $\ints[\zeta_d]$-linear direct summand of the cohomology of degree $d - 2$ of the Dwork family. Conversely, any irreducible summand of $H^{d-2}(Y_t, \ints[\frac{1}{d},\zeta_d])$ is of this form. 

\begin{num}
	In \S 2 we explain how to describe properties of the family of motives $M(\alpha_1, \dots, \alpha_n ; \beta_1, \dots, \beta_n)_t$ in terms of simple combinatorics of the tuple $(\alpha_1, \dots, \alpha_n ; \beta_1, \dots, \beta_n)$. For instance, the Jordan normal form of the monodromy operator at $t = \infty$ has block sizes $\# \{ i \in \{1, \dots, n\} : \alpha_i = \alpha \}$ for $\alpha \in \ints/d\ints$. This treats property (UM). We also state explicit combinatorial criteria for properties (BM) and (R).
\end{num}

\begin{num}
	The work of Harris--Shepherd-Barron--Taylor \cite{HSBT} first introduced the idea of using the Dwork family for potential automorphy. They considered $d$ odd, $n = d - 1$ and used the hypergeometric parameter
	\[
	(0, \dots, 0 ; 1, 2, \dots, d - 1).
	\]
	The corresponding family of motives has maximal unipotent monodromy at $t = \infty$ and the geometric monodromy group is Zariski dense in $\Sp_n$. The Hodge numbers are $h^{i,j} = 1$ for $i + j = d - 2$ and $i,j \geq 0$.
\end{num}

\begin{num}
	In later works, Barnet-Lamb \cite{BL_even}, \cite{BL_odd} also considered parameters of the form 
	\[
	(0, \dots, 0 ; \beta_1, \dots, \beta_n)
	\]
	for any degree $d$ and suitable choices of $\beta_i \in \ints/d\ints$. This gives rise to monodromy groups which are dense in $\SL_n$, i.e. satisfy (BM). It turns out that parameters of this form also always satisfy (R).
	Subsequent variants \cite{BLGHT}, \cite{qian_potential} introduced even more parameters of this form to prove their potential automorphy theorems. All of these examples share the property that the monodromy at $t = \infty$ is maximally unipotent, i.e. they do not necessarily satisfy (UM).
\end{num}

\begin{num}
	To realise nilpotent operators with Jordan block sizes $p_1, \dots, p_r$ in property (UM), we consider 
	parameters of the form
	\[
	(\underbrace{0 , \dots, 0}_{p_1}, \underbrace{\gamma_2, \dots, \gamma_2}_{p_2}, \dots, \underbrace{\gamma_r, \dots, \gamma_r}_{p_r} ; \beta_1, \dots, \beta_n).
	\]
	with $\gamma_i, \beta_i \in \ints/d\ints$. Unfortunately, it is quite rare for such a parameter to have regular Hodge--Tate weights (property (R)). Although it is easy to check for any given parameter if (R) is satisfied.
	
	Indeed, Proposition \ref{hodge_numbers_prop} implies that a parameter $(\alpha_i ; \beta_j)$ satisfies (R) if and only if for every $s \in (\ints/d\ints)^\times$, the following $n$ integers are pairwise distinct
	\[
	\sum_{i} [s \beta_j - s \alpha_i] - \sum_{i} [ s \beta_j - s \beta_i] \qquad j = 1, \dots, n,
	\]
	where $[x] \in \ints$ denotes the unique representative of $x \in \ints/d\ints$ satisfying $0 \leq [x] \leq d - 1$. Alternatively, this condition can be phrased in terms of the cyclic order on $\ints/d\ints$ as in \cite[\S 5]{roberts_villegas}: For every $s \in (\ints/d\ints)^\times$, the zigzag procedure for $(s \alpha_i ; s \beta_j)$ should have at most one local maximum and one local minimum.
	
	Since scaling by $s \in (\ints/d\ints)^\times$ does not interact well with the function $[ \cdot ]$, we do not know a good way of generating parameters which satisfy (R).
\end{num}

\begin{num}
	For $n \leq 5$ and any nilpotent operator, we have computationally found parameters satisfying properties (BM), (R), (UM) (see table \ref{special_hypergeos} at the end of this paper). For example for $n = 4$, and 
	\[
		N = \begin{bmatrix}
			0 & 1 & 0 & 0 \\
			0 & 0 & 1 & 0 \\
			0 & 0 & 0 & 0 \\
			0 & 0 & 0 & 0 
		\end{bmatrix}
	\] 
	we found the parameter modulo $18$ given by
	\[
		(0, 0, 0, 3 ; 4, 11, 16, 17).
	\]
	We also computed some examples with $n = 6$. The method of computation is presented in \S 4.
	It would be interesting to improve the brute-force algorithm or to find more conceptual ways to satisfy the combinatorial properties.
\end{num}

\subsection{Galois Representations} \label{galois_intro} 
Applying the method of \cite{HSBT} to our families of motives, we can prove a potential automorphy theorem with additional control on the monodromy operator. To state it we recall a few technical terms. 

\begin{num} 
	For a number field $F$, we let $\mathbb{A}_F$ be the ring of adeles over $F$ and we let $G_F$ denote the absolute Galois group of $F$. Moreover, we let $\mu_d$ be the set of $d$th roots of unity in some algebraic closure of $F$ and we let $\chi_{cyc} : G_{F} \to \ints_\ell^\times$ denote the $\ell$-adic cyclotomic character. For an isomorphism $\iota : \overline{\rationals}_\ell \cong \complex$ and a continuous Galois representation $\rho : G_{F} \to \GL_n(\overline{\rationals}_\ell)$, we refer to section \S \ref{WD_section} for the definition and properties of the Weil--Deligne representations $WD_{\iota}(\rho \rvert_{F_v})$, where $v$ is a prime of $F$. Finally, we shall denote by $\overline{\rho} : G_{F} \to \GL_n(\overline{\mathbb{F}}_\ell)$ the reduction of $\rho$ modulo $\ell$, which is a well-defined semisimple representation.
\end{num}

\begin{num}
	We say that an automorphic representation $\Pi$ of $\GL_n(\mathbb{A}_F)$ is cohomological if there exists an algebraic representation $V$ of $\GL_n/F$ such that the infinitesimal character of $\Pi_\infty$ agrees with that of $V$.
	For an isomorphism $\iota : \overline{\rationals}_\ell \cong \complex$ and a cohomological cuspidal automorphic representation, we let $r_{\iota}(\Pi) : G_F \to \GL_n(\overline{\rationals}_\ell)$ be the Galois representation defined in \cite{harris_lan_taylor_thorne}. Moreover, we will need the notion of $\iota$-ordinarity of $\Pi$ at a prime $v$ of $F$, which is defined in \cite[5.3]{geraghty19}.
\end{num}
	
	\begin{theorem}[ = Theorem \ref{potential_automorphy_thm}] 
	Let $(\alpha_1, \dots, \alpha_n ; \beta_1, \dots, \beta_n)$ be a hypergeometric parameter of dimension $n$ modulo $d$ appearing in table \ref{special_hypergeos}.
	There exists an integer $C > 0$ with the following property. Given the following objects:
\begin{itemize}
	\item a CM field $F$ containing $\mu_d$,
	\item a prime $p \nmid d$,
	\item a prime $\ell \nmid p C$ such that $\ell \equiv 1 \pmod{d}$, 
	\item a continuous semisimple representation $\overline{\rho} : G_{F} \to \GL_n(\mathbb{F}_\ell)$ such that $\det \overline{\rho} = \overline{\chi}_{cyc}^{-n(n-1)/2}$,
	\item a finite extension $F^{avoid}/F$,
\end{itemize}
there exists a finite totally real extension $F_2/\rationals$, linearly disjoint from $F^{avoid}$ over $\rationals$ such that for $F' = F F_2$, there exists an isomorphism $\iota : \overline{\rationals}_\ell \cong \complex$ and a cuspidal automorphic representation $\Pi$ of $\GL_n(\mathbb{A}_{F'})$ which is $\iota$-ordinary of weight $0$ at all places above $\ell$
such that 
\[
\overline{r_{\iota}(\Pi)} = \overline{\rho} \rvert_{G_{F'}}
\]
and for all places $v \mid p$ of $F'$
\[
WD(r_{\iota}(\Pi) \rvert_{G_{F_v'}})^{F-ss}
\]
is generic with monodromy operator having Jordan blocks $J_\alpha$ of size $\#\{i : \alpha_i = \alpha\}$ for $\alpha \in \ints/d\ints$. If $d$ is odd, then there also exists another cuspidal automorphic representation $\Pi_2$ of $\GL_n(\mathbb{A}_{F'})$ with the same properties as $\Pi$ except that the $v$-adic monodromy operator is nilpotent of rank $1$.
	\end{theorem}

	\begin{num}
		This theorem can be seen as a refinement of the main theorem of \cite{qian_potential}. Namely, we prove potential automorphy together with a specified choice of local monodromy. The proof of the theorem closely follows the methods of \cite{HSBT}. However, one major new ingredient is a theorem of Scholze \cite{scholze_perfectoid}, which proves the weight-monodromy conjecture for hypersurfaces. We use it to obtain the genericity in the conclusion of the theorem.
		
		We require stronger assumptions on $\ell$ and the coefficient field of $\overline{\rho}$ as in \cite{qian_potential}, but these could be removed if we had more examples of motives satsifying the hypotheses of proposition \ref{big_monodromy_finite_prop} with sufficiently small $U$. Our assumption on $\ell$ also allows us to easily prove the ordinarity of the Galois representations corresponding to our families of motives. (The proof of ordinarity from \cite{qian_potential} does not apply in our setting.)
	\end{num}
	
	\begin{num}
		As a consequence of theorem \ref{potential_automorphy_thm}, we prove the existence of non conjugate self-dual automorphic representations with local monodromy operator $N$ (see corollary \ref{automorphic_existence_cor}). 
		We do not know how to construct such representations by purely automorphic methods. For groups which have discrete series, there are such methods \cite[Thm 5.7]{shin12}, but this does not apply to non conjugate self-dual automorphic representations of $\GL_n$ for $n \geq 3$.
		
		As another application, we generalise the method introduced by Allen--Newton \cite{allen_newton} to prove the following local-global compatibility theorem.
	\end{num}
	
	\begin{theorem}[ = Theorem \ref{local_global_thm}]
		Let $n \in \{3,4,5\}$ and $d_3 = 9$, $d_4 = 36$, $d_5 = 168$.
There exists an integer $C > 0$ with the following property.
Given the following objects:
\begin{itemize}
	\item a CM field $F$ containing $\mu_{d_n}$,
	\item a prime $\ell$ such that $\ell \nmid C$ and $\ell \equiv 1 \pmod{d_n}$ and an isomorphism $\iota : \overline{\rationals}_\ell \cong \complex$,
	\item a cuspidal cohomological automorphic representation $\Pi$ of $\GL_n(\mathbb{A}_F)$ which is $\iota$-ordinary at all places above $\ell$, 
\end{itemize}
such that
\begin{itemize}
	\item $\overline{r_{\iota}(\Pi)}(G_{F(\mu_\ell)})$ is enormous \cite[6.2.28]{10author} and $\overline{r_{\iota}(\Pi)}$ is absolutely irreducible and decomposed generic \cite[4.3.1]{10author},
	\item there exists $\sigma \in G_{F} \setminus G_{F(\mu_\ell)}$ such that $\overline{r_{\iota}(\Pi)}(\sigma)$ is a scalar,
	\item there exists $\gamma \in \GL_n(\overline{\mathbb{F}}_\ell)$ such that $\gamma \overline{r_\iota(\Pi)}(G_{F}) \gamma^{-1} \subset \GL_n(\mathbb{F}_\ell)$,
	\item $\det \overline{r_{\iota}(\Pi)} = \overline{\chi}_{cyc}^{-n(n-1)/2}$,
\end{itemize}
we have
\[
WD(r_{\iota}(\Pi) \rvert_{G_{F_v}})^{F-ss} \cong \operatorname{rec}_{F_v}(\Pi_v |\det |^{(1-n)/2}),
\]
for all places $v \nmid d_n \ell $.
\end{theorem}

	\begin{num}
	During the preparation of the present article, Matsumoto \cite{matsumoto2024potential}
	has proven a very general local-global compatibility theorem which works for all $n$ and includes our theorem \ref{local_global_thm} as a special case. Their proof is based on the Harris tensor product trick and requires neither our potential automorphy theorem nor any new families of motives. 
	
	Previously, Varma \cite{varma_local_global} has proven local-global compatibility up to semisimplicity. For $n = 2$, Allen--Newton \cite{allen_newton} and Yang \cite{yang_monodromy} have proven many cases of full local-global compatibility. Their work was the inspiration for theorem \ref{local_global_thm}.
	
	When our theorem or Matsumoto's theorem applies, then it implies that condition (d) of \cite[Thm 5.3.5]{acampo_rigidity} is satisfied. This was the original motivation for this paper.
	\end{num}
	
	\subsection*{Acknowledgements} 
	I am grateful to my advisor James Newton for his encouragement to write up this paper and for many interesting discussions on the subject. I thank Alexander Bertoloni Meli, Lie Qian, Peter Scholze, Richard Taylor and Pol van Hoften for helpful conversations and correspondence. I thank Dmitri Whitmore for pointing out an error in an earlier version of this manuscript, concerning the tables of hypergeometric parameters.

	\section{The Dwork Family}
	
	\begin{num} 	
		In this section we recall the properties of the (unweighted) Dwork family and define the families of motives which are the subject of this paper. 
		Many authors have written about this subject, but our main references are \cite{deligne_milne_ogus_shih}, \cite{HSBT}, \cite{katz_another_look}, \cite{kloosterman_fermat}.
		
		Let $d > 2$ be an integer, set $\mathcal{O}_d = \ints[ \frac{1}{d}, e^{2\pi i/ d}] \subset \complex$ and let
		\[
		A_d = \mathcal{O}_d\left[t, \frac{1}{t^d - 1} \right].
		\]
	\end{num}
	
	\begin{definition}
		Let $d > 2$ be an integer. We define the projective $A_d$-scheme 
		\[
			\mathcal{Y} := \Proj A_d [x_0, \dots, x_{d-1}] /(F_t),
		\]
		where $F := \sum_{i = 0}^{d - 1} x_i^{d}$ and $F_t = F - t d x_0 \cdots x_{d-1}$. Denote the structure morphism by 
		\[	f : \mathcal{Y} \to \Spec A_d. \]
		One can check that $f$ is smooth using the Jacobian criterion.
	\end{definition}

	\begin{num}
		Let $\mu_d = \{z \in \complex : z^d = 1\} \subset \complex$ denote the group of $d$th roots of unity under multiplication. The group 
		\[
		H_0 := \{ (\zeta_0, \dots, \zeta_{d-1}) \in \mu_{d}^d : \zeta_0 \cdots \zeta_{d-1} = 1 \} / \{ (\zeta, \dots, \zeta) : \zeta \in \mu_d \}
		\] 
		acts on $\mathcal{Y}$ by 
		\[
			(\zeta_0, \dots, \zeta_{d-1}) \cdot (x_0 : \dots : x_{d-1}) = (\zeta_0 x_0 : \dots : \zeta_{d-1} x_{d-1}),
		\] 
		leaving $f$ invariant. On the fibre $\mathcal{Y}_0 := \mathcal{Y} \times_{\Spec A_d} \Spec A_d/(t)$, the same formula defines an action of the bigger group
		\[
			H := \mu_d^d / \{ (\zeta, \dots, \zeta) : \zeta \in \mu_d \}.
		\]
\end{num}

\begin{num}
	For $\underline{a} = (a_0, \dots, a_{d-1}) \in (\ints/d\ints)^d$ such that $\sum_{i} a_i = 0$ we define the character
	\[
	\chi_{\underline{a}} : H \to \mathcal{O}_d^\times : (\zeta_0, \dots, \zeta_{d-1}) \mapsto \prod_{i = 0}^{d-1} \zeta_i^{a_i}.
	\]
	Every character of $H$ is of this form for a unique $\underline{a}$ and $\chi_{\underline{a}} \rvert_{H_0} = \chi_{\underline{a}'} \rvert_{H_0}$ if and only if there exists $b \in \ints/d \ints$ such that 
	\[
	\underline{a}' = \underline{a} + b := (a_0 + b, a_1 + b, \dots, a_{d-1} + b).
	\]
\end{num}

\begin{num}
	If a group $G$ acts on the left on a topological space $X$, we define the induced left group action on its cohomology $H^\bullet(X, \ints)$ by 
	\[
	g \cdot \omega := (g^{-1})^* \omega
	\]
	for $g \in G$ and $\omega \in H^\bullet(X, \ints)$. If $V$ is a $\mathcal{O}_d[H_0]$-module we let $V_{\underline{a}}$ be the $\chi_{\underline{a}} \rvert_{H_0}$-isotypic part of $V$. For example
	\[
		H^{d-2}(\mathcal{Y}_{t}, \complex)_{\underline{a}} = \{ \omega \in H^{d-2}(\mathcal{Y}_{t}, \complex) : \forall h \in H_0, \ (h^{-1})^* \omega = \chi_{\underline{a}}(h) \omega\}.
	\]
\end{num}

\subsection{Statements} Here we define our families of motives and state their most important properties. The proofs follow in the subsequent sections.

\begin{definition} \label{hypergeo_param_def}
	Let $d > n > 0$ be integers. A hypergeometric parameter of dimension $n$ modulo $d$ is a tuple $(\alpha_1, \dots, \alpha_n; \beta_1, \dots, \beta_n)$ of elements of  $\ints/d\ints$ such that
	\begin{itemize}
		\item $\sum_i \alpha_i - \sum_i \beta_i = \binom{d}{2}$,
		\item $\alpha_i \neq \beta_j$ for $i, j \in \{1, \dots, n\}$,
		\item $\beta_i \neq \beta_j$ for $i \neq j$.
	\end{itemize}
\end{definition}

\begin{definition} \label{motive_def}
	Let $(\alpha_1, \dots, \alpha_n ; \beta_1, \dots, \beta_n)$ be a hypergeometric parameter of dimension $n$ modulo $d$. Let
	\[
		\underline{a} = (-\alpha_1, \dots, - \alpha_n, s_0, \dots, s_{d - n - 1}) \in (\ints/d\ints)^d,
	\]
	where
	$\{s_1, \dots, s_{d - n - 1}\} = \{0, \dots, d - 1\} \setminus \{- \beta_i\}$. Since $(\alpha_i ; \beta_j)$ is a hypergeometric parameter, we have $\sum_{i} a_i = 0$.
	Suppose $K$ is a field containing the $d$th roots of unity and $t \in K \setminus \mu_d$. We define
	\begin{itemize}
		\item $M(\alpha_1, \dots, \alpha_n; \beta_1, \dots, \beta_n)_{dR, t} := H^{d-2}_{dR}(\mathcal{Y}_t/ K)_{\underline{a}}^{(prim)}$,
		\item $M(\alpha_1, \dots, \alpha_n; \beta_1, \dots, \beta_n)_{\lambda, t} = H^{d-2}_{et}(\mathcal{Y}_{t, \overline{K}}, \mathcal{O}_{d,\lambda})_{\underline{a}}^{(prim)}$
		for primes $\lambda < \mathcal{O}_d$,
		\item $M(\alpha_1, \dots, \alpha_n ; \beta_1, \dots, \beta_n)_{B, t} :=  H^{d-2}_{sing}(\mathcal{Y}_{t}(\complex), \mathcal{O}_d)_{\underline{a}}^{(prim)}$ if $K \subset \complex$.
	\end{itemize}
	Here $(prim)$ refers to the primitive part of the cohomology and is only relevant for the definition when $n$ is odd and $d = n + 1$.
\end{definition}

\begin{prop}[= Proposition \ref{hodge_numbers_prop}]
	Let $(\alpha_1, \dots, \alpha_n ; \beta_1, \dots, \beta_n)$ be a hypergeometric parameter of dimension $n$ modulo $d$. For every $t \in \complex \setminus \mu_d$, we have:
	\begin{itemize}
		\item $M_{*,t}$ is a free module of rank $n$ for $* \in \{dR, B, \lambda\}$.
		\item The dimension of $\operatorname{gr}_{\mathrm{Hodge}}^{p} (M_{dR,t}) $ is equal to the number of indices $j \in \{1, \dots, n\}$ such that
		\[
		d(p + 1) = \binom{d}{2} + \sum_{i = 1}^{n} [\beta_j - \alpha_i]  - \sum_{i = 1}^n [\beta_j - \beta_i],
		\]
		where $[x]$ denotes the unique representative of $x$ such that $0 \leq [x] \leq d - 1$.
	\end{itemize}
\end{prop}

\begin{prop}[= Proposition \ref{unipotent_monodromy_prop}]
	Let $(\alpha_1, \dots, \alpha_n ; \beta_1, \dots, \beta_n)$ be a hypergeometric parameter of dimension $n$ modulo $d$ and 
\[
M = M(\alpha_1, \dots, \alpha_n ; \beta_1, \dots, \beta_n).
\] 
For $t \in \complex \setminus \mu_d$ let $\gamma_\infty \in \pi_1(\complex \setminus \mu_d, t)$ be a simple loop based at $t$ going around $\infty$. Then $\gamma_{\infty}$ acts on $M_{B,t}$ as a unipotent operator with Jordan blocks $J_\alpha$ of size $\#\{i : \alpha_i = \alpha\}$ for $\alpha \in \{\alpha_i\}$.
Similarly, let $\gamma_1$ denote a simple loop around $1$. Then $\gamma_{1}$ acts as a unipotent operator on $M_{B,t}$ and
the rank of $\gamma_{1} - \id$ is $1$ if $d$ is odd and $0$ if $d$ is even.
\end{prop}

\begin{prop}[= Proposition \ref{big_monodromy_prop}]
	Let $(\alpha_1, \dots, \alpha_n ; \beta_1, \dots, \beta_n)$ be a hypergeometric parameter of dimension $n$ modulo $d$ and $t \in \complex \setminus \mu_d$. If the following hypotheses are satisfied 
\begin{itemize}
	\item $|\{\alpha_1, \dots, \alpha_n\}| < n$,
	\item $\{\beta_1, \dots, \beta_n\}$ does not form an arithmetic progression,
	\item there is no non-zero $s \in \ints/d\ints$ such that
	\[
	\{\alpha_i + s\} = \{\alpha_i\} \quad \wedge \quad \{\beta_i + s\} = \{\beta_i\},
	\]
	\item there is no $s \in \ints/d\ints$ such that
	\[
	\{-\alpha_i - s\} = \{\alpha_i + s\} \quad \wedge \quad \{-\beta_i -s\} = \{\beta_i + s\},
	\]
\end{itemize}
then the identity component of the Zariski closure of the image of the geometric monodromy representation
\[
\pi_1(\complex \setminus \mu_d, t) \to \GL(M(\alpha_1, \dots, \alpha_n ; \beta_1, \dots, \beta_n)_{B,t} \otimes_{\mathcal{O}_d} \complex)
\]
contains $\SL_n(\complex)$.
\end{prop}

\begin{prop}[ = Proposition \ref{big_monodromy_finite_prop}]
	Let $(\alpha_1, \dots, \alpha_n ; \beta_1, \dots, \beta_n)$ be a hypergeometric parameter of dimension $n$ modulo $d$ and $t \in \complex \setminus \mu_d$ satisfying the hypotheses of the previous proposition. Moreover, let $U,V < (\ints/ d\ints)^\times$ be subgroups such that $\{1\} = U \cap V$, $UV = (\ints/d\ints)^\times$ and for $s \in (\ints/d\ints)^\times$ the equalities
\begin{align*}
	\{ s\alpha_i - s\alpha_j : 1 \leq i,j \leq n \} &= \{\alpha_i - \alpha_j : 1 \leq i,j \leq n\} \\ 
	\{ s\beta_i - s\beta_j : 1 \leq i,j \leq n \} &= \{\beta_i - \beta_j : 1 \leq i,j \leq n\}
\end{align*}
imply that $s \in U$. Let $\lambda < \mathcal{O}_d$ be a sufficiently large prime of residue characteristic $\ell$ such that $\ell \in V$. Then the image of
\[
\pi_1(\Spec A_d(\complex), t) \to \GL(M(\alpha_i ; \beta_j)_{B,t} \otimes_{\mathcal{O}_d} \mathcal{O}_d/\lambda) 
\]
contains $\SL_n(\mathcal{O}_d/\lambda)$.
\end{prop}

\begin{num}
	For $a \in (\ints/d\ints) \setminus \{0\}$ and a divisor $k \mid d$ such that $k a \neq 0$, we define the functions $\delta_a, \epsilon_{k, a} : (\ints/d\ints) \setminus \{0\} \to \ints$ as follows
	\begin{align*}
		\delta_a(x) = \begin{cases}
			1 & x = a \\
			0 & \text{ otherwise }
		\end{cases}
	\end{align*} 
	and
	\[
		\epsilon_{k,a}(x) = \delta_{- k a} + \sum_{0 \leq j \leq k - 1} \delta_{a + j d/k}.
	\]
	We let $E(d)$ denote the set of functions consisting of $\epsilon_{1,a}$ and $\epsilon_{p,a}$, where $p$ is a prime divisor of $d$.
\end{num}

\begin{prop}[= Proposition \ref{determinant_nth_power_prop}]
	Let $(\alpha_1, \dots, \alpha_n ; \beta_1, \dots, \beta_n)$ be a hypergeometric parameter of dimension $n$ modulo $d$ and $F$ a number field containing the $d$th roots of unity. Let $t \in F \setminus \mu_d$ and $\underline{c} \in (\ints/d\ints)^3$ such that $c_0 + c_1 + c_2 = 0$ and either $c_i \neq 0$ for $i =0,1,2$ or $c_0 = c_1 = c_2 = 0$.
Suppose the following hypotheses are satisfied:
\begin{itemize}
	\item If $d$ is even, then $1 - t^d$ is a square.
	\item For all $s \in (\ints/d\ints)^\times$, the following $n$ integers are pairwise distinct
	\[
	\sum_{i} [s \beta_j - s \alpha_i] - \sum_{i} [s \beta_j - s \beta_i] \qquad j = 1, \dots, n.
	\]
	\item The following quantity is independent of $s \in (\ints/d\ints)^\times$
	\[
	\sum_{i,j} [s \beta_j - s \alpha_i] - \sum_{i,j} [s \beta_j - s \beta_i] + n \sum_{i = 0}^2[s c_i].
	\]
	\item For each $s \in (\ints/d\ints)^\times$ we have
	\[
	\sum_{i,j} [s \beta_j - s \alpha_i] = n \sum_i [s \beta_i - s \alpha_i].
	\]
	\item There are integers $x_{\epsilon}$ for $\epsilon \in E(d)$ such that
	\[
	n + \sum_{i,j} \delta_{\beta_j - \alpha_i} - \sum_{i \neq j} \delta_{\beta_j - \beta_i} + n \sum_{i = 0}^2 \delta_{c_i} = \sum_{\epsilon \in E(d)} x_{\epsilon} \epsilon.
	\]
	Let $b_1$ and $b_p$ be the integers defined in corollary \ref{gamma_degree_cor}. The integer
	$\varphi(\operatorname{lcm}(2 b_1, d))/\varphi(d)$ is coprime to $n$ and $b_p$ is coprime to $n$.
\end{itemize}
Then there exists a character $\psi : G_{\rationals(\mu_d)} \to \mathcal{O}_{d,\lambda}^\times$, such that 
\[
\det(M(\alpha_i ; \beta_j)_{\lambda, t} \otimes \psi) = \chi_{cyc}^{-n(n - 1)/2}.
\]
\end{prop}

\begin{prop} \label{computed_motives_prop}
	Let $(\alpha_1, \dots, \alpha_n ; \beta_1, \dots, \beta_n)$ be a hypergeometric parameter of dimension $n$ modulo $d$ appearing with $U < (\ints/d\ints)^\times$ in table \ref{special_hypergeos} at the end of the paper. Let $V < (\ints/d\ints)^\times$ be a subgroup such that $\{1\} = U \cap V$, $UV = (\ints/d\ints)^\times$. Then the family of motives $M = M(\alpha_1, \dots, \alpha_n; \beta_1, \dots, \beta_n)$ enjoys the following properties 
	\begin{enumerate}
		\item[(BM)] For $t \in \complex \setminus \mu_d$, the Zariski closure of the image of
		\[ \pi_1(\complex \setminus \mu_d, t) \to \GL(M_{B,t} \otimes_{\mathcal{O}_d} \complex)
		\]
		contains $\SL_n(\complex)$. Moreover, there exists a constant $C = C(d, \alpha_i; \beta_i)$ such that if $\lambda \neq \lambda' < \mathcal{O}_d$ are primes above $\ell, \ell'$ not dividing $C$ and $\ell, \ell' \in V$, then
		the image of 
		\[
			\pi_1(\complex \setminus \mu_d, t) \to \GL( M_{B,t} \otimes_{\mathcal{O}_d} \mathcal{O}_d/\lambda\lambda')
		\]
		contains $\SL_n(\mathcal{O}_d/\lambda\lambda')$.
		\item[(R)] For $t \in \complex \setminus \mu_d$, the Hodge numbers of $M_{dR,t} \otimes_{\iota} \complex$ are $\leq 1$ for every embedding $\iota : \mathcal{O}_d \to \complex$.
		\item[(UM)]  For $t \in \complex \setminus \mu_d$, let $\gamma_\infty$ denote a simple loop around $\infty$ based at $t$. Then $\gamma_{\infty, *}$ acts on $M_{B,t}$ as a unipotent operator with Jordan blocks $J_\alpha$ of size $\#\{i : \alpha_i = \alpha\}$ for $\alpha \in \{\alpha_i\}$.
		\item[(D)] Let $F$ be a number field containing the $d$th roots of unity and $t \in F \setminus \mu_d$ such that if $d$ is even, then $1 - t^d$ is a square. The Galois group $G_F$ acts on  $\bigwedge^n M_{\lambda, t}$ as $\chi_{cyc}^{-n(n - 1)/2} \psi^{-n}$, for some continuous character $\psi : G_{F} \to \mathcal{O}_{d,\lambda}^\times$ which is independent of $t$.
	\end{enumerate}
\end{prop}

\begin{proof}
	With the previous propositions this proof becomes a purely computational task which has been completed by the computer program described in \S 4.
\end{proof}

\subsection{The Method of Griffiths} The work of Griffiths \cite{griffiths_periods} conveniently describes the Hodge theory and periods of projective hypersurfaces such as $\mathcal{Y}_t$. Thus, we recall a few aspects of Griffiths' theory and then apply it to the local system $M(\alpha_i; \beta_j)_{B, t}$ and the variation of Hodge structures given by $M(\alpha_i; \beta_j)_{dR,t}$.

\begin{num}
	For $t \in \Spec A_d(\complex) = \complex \setminus \mu_d$, there is the tube map
	\[
	\tau : H_{d-2}(\mathcal{Y}_t(\complex), \mathcal{O}_d)^{(prim)} \to H_{d-1}(\mathbb{P}^{d - 1}(\complex) \setminus \mathcal{Y}_t(\complex), \mathcal{O}_d)
	\]
	defined in \cite[\S 3]{griffiths_periods}. It is an isomorphism by \cite[Prop 3.5]{griffiths_periods}. 
\end{num}

\begin{num}
	We describe the cohomology of the smooth projective hypersurface $\mathcal{Y}_t \subset \mathbb{P}^{d-1}$ in terms of the de Rham cohomology of the affine variety $\mathbb{P}^{d-1} \setminus \mathcal{Y}_t$ as follows. Define the vector spaces
	\[
	A_l^q(\mathcal{Y}_t) := \{ \omega \in H^0(\mathbb{P}^{d-1}_{\complex} \setminus \mathcal{Y}_{t,\complex}, \Omega^q_{\mathbb{P}^{d-1}}) : F_t^l \omega \text{ is regular  on } \mathbb{P}^{d-1}_{\complex} \}
	\]
	and let $F^\bullet H^{d - 2}(\mathcal{Y}_t, \complex)$ denote the Hodge filtration. Then \cite[Thm 8.1]{griffiths_periods} implies that we have linear maps
	\[
	A^{d-1}_{l}(\mathcal{Y}_t)/ d A^{d-1}_{l-1}(\mathcal{Y}_t) \xrightarrow{R} F^{d - 1 - l} H^{d - 2}(\mathcal{Y}_t(\complex), \complex)
	\]
	for all $l \geq 1$
	which satisfy the following characterising property.
	If $\gamma \in H_{d-2}(\mathcal{Y}_t(\complex), \mathcal{O}_d)$, then
	\[
	\frac{1}{2 \pi i} \int_{\tau(\gamma)} \omega = \int_{\gamma} R(\omega).
	\]
	In particular, for all $h \in H_0$ we have $R(h^* \omega) = h^* R(\omega)$.
	We moreover know that for $l = 1, 2, \dots, d - 1$, the map
	\[
	A^{d-1}_{l}(\mathcal{Y}_t)/ d A^{d-1}_{l-1}(\mathcal{Y}_t) \xrightarrow{R} F^{d - 1 - l} H^{d - 2}(\mathcal{Y}_t(\complex), \complex)
	\]
	is injective and the image is the primitive part of the cohomology \cite[Thm 8.3]{griffiths_periods}.
\end{num}

\begin{num} \label{griffiths_explicit_iso_para}
	Let $\complex[x_0, \dots, x_{d-1}]_{k}$ denote the vector space of homogeneous degree $k$ polynomials. By \cite[Cor 2.11]{griffiths_periods}, we have an isomorphism 
	\[
		\complex[x_0, \dots, x_{d-1}]_{dl} \to A_{l + 1}^{d-1}(\mathcal{Y}_t) : P \mapsto \frac{P}{F_t^{l + 1}} \Omega
	\]
	for $l = 0, \dots, d- 2$,
	where
	\[
	\Omega = \sum_{i = 0}^{d-1} (-1)^{i} x_i dx_{0} \wedge \dots \wedge \widehat{d x_{i}} \wedge \dots dx_{d-1}.
	\]
	Let $J_t = (\frac{\partial}{\partial x_0} F_t, \dots, \frac{\partial}{\partial x_{d-1}}F_t ) < \complex[x_0, \dots, x_{d-1}]$ be the Jacobian ideal of $F_t$.
	Then \cite[Prop 4.6]{griffiths_periods} implies that the map 
	\begin{align*}
		\left( \complex[x_0, \dots, x_{d-1}] /J_t \right)_{dl} & \to H^{d-2 - l, l}(\mathcal{Y}_t(\complex)) \\
		P & \mapsto R \left( \frac{P}{F_t^{l + 1}} \Omega \right)
	\end{align*}
	is a well-defined injection whose image is the primitive part of the cohomology.
\end{num}

\subsection{The Cohomology at \texorpdfstring{$t = 0$}{t = 0}}
	The cohomology of $\mathcal{Y}_0$ is described in detail in \cite[I. \S 7]{deligne_milne_ogus_shih}. We recall parts of this description here, since we will need it for our study of the periods of $\mathcal{Y}_t$ for $t \neq 0$.
	
\begin{num} \label{fermat_para}
	For $t = 0$, we have the particularly simple Jacobian ideal 
	\[
		J_0 = (x_0^{d-1}, \dots, x_{d-1}^{d-1}).
	\]
	Thus, $\complex[x_0,\dots, x_{d-1}]/J_0$ has the monomial basis
	\[
		\{ x^{\underline{e}} = x_0^{e_0} x_1^{e_1} \dots x_{d-1}^{e_{d-1}} : \forall i = 0, \dots, d - 1, \ e_i < d - 1 \}.
	\]
	Moreover, if $h = (\zeta_0, \dots, \zeta_{d-1}) \in H$, then 
	\[
		(h^{-1})^* \frac{ x^{\underline{e}} \Omega}{F^{l + 1}} = \prod_{i = 0}^{d-1} \zeta_i^{-e_i - 1} \frac{ x^{\underline{e}} \Omega}{F^{l + 1}}.
	\]
	Consequently, paragraph \ref{griffiths_explicit_iso_para} implies that for $\underline{a} \in (\ints/d\ints)^d$ such that $\sum_{i = 0}^{d-1} a_i = 0$, we have the decomposition
	\[
		H^{d - 2 - l,l}(\mathcal{Y}_0(\complex))_{\underline{a}}^{(prim)} = \bigoplus_{\underline{e} \in B_{\underline{a}}^{l}} R\left( \frac{x^{\underline{e}} \Omega}{F^{l + 1}} \right) \complex,
	\]
	where 
	\[
		B_{\underline{a}}^{l} = \left\{ \underline{e} \in \{0, \dots, d-2\}^{d} : \exists b \in \ints/d\ints,  \begin{array}{l} -e_i - 1 \equiv a_i + b \pmod{d} \\ \sum_{i = 0}^{d-1} e_i = dl \end{array} \right\}.
	\]
	Equivalently, the dimension of $H^{d- 2 - l, l}(\mathcal{Y}_0(\complex))_{\underline{a}}^{(prim)}$ is equal to the number of solutions to the equation
	\[
		d l = \sum_{i = 0}^{d-1} [b - a_i]
	\] 
	with $b \in \ints/d\ints \setminus \{-a_0, \dots, - a_{d-1}\}$, where $[x]$ denotes the unique representative of $x \in \ints/d\ints$ in $\{0,1,\dots, d-1\}$.
	Summing over $l = 0, \dots, d-2$ we find that
	\begin{align*}
		\dim H^{d-2}(\mathcal{Y}_0(\complex), \complex)_{\underline{a}}^{(prim)} &= | \ints/d\ints \setminus \{-a_0, \dots, -a_{d-1} \}|.
	\end{align*}
	As a consequence, we obtain
\end{num}

\begin{prop} \label{hodge_numbers_prop}
	Let $(\alpha_1, \dots, \alpha_n ; \beta_1, \dots, \beta_n)$ be a hypergeometric parameter of dimension $n$ modulo $d$. For every $t \in \complex \setminus \mu_d$, we have:
	\begin{itemize}
		\item $M_{*,t}$ is a free module of rank $n$ for $* \in \{dR, B, \lambda\}$.
		\item The dimension of $\operatorname{gr}_{\mathrm{Hodge}}^{p} (M_{dR,t}) $ is equal to the number of indices $j \in \{1, \dots, n\}$ such that
		\[
		d(p + 1) = \binom{d}{2} + \sum_{i = 1}^{n} [\beta_j - \alpha_i]  - \sum_{i = 1}^n [\beta_j - \beta_i],
		\]
		where $[x]$ denotes the unique representative of $x$ such that $0 \leq [x] \leq d - 1$.
	\end{itemize}
\end{prop}

\begin{proof}
	The singular cohomology of projective hypersurfaces is torsion-free, thus by the usual comparison theorems it suffices to prove the first part for $M_{B, t} \otimes_{\mathcal{O}_d} \complex$. Moreover, the dimension of $M_{B,t} \otimes_{\mathcal{O}_d} \complex$ is locally constant by the theorem of Ehresmann so it suffices to prove the claim for $t = 0$. By definition, $M_{B,0} = H^{d-2}(\mathcal{Y}_0(\complex), \mathcal{O}_d)_{\underline{a}}$, where 
		\[
	\underline{a} = (-\alpha_1, \dots, - \alpha_n, s_0, \dots, s_{d - n - 1}) \in (\ints/d\ints)^d
	\]
	and $\{s_1, \dots, s_{d- n - 1}\} = \{0, \dots, d-1\} \setminus \{- \beta_i\}$.
	Hence, 
	\[
	\ints/d\ints \setminus \{-a_0, \dots, - a_{d-1}\} = \{\beta_1, \dots, \beta_n\}
	\] 
	has cardinality $n$ and paragraph \ref{fermat_para} implies that 
	\[
		\dim M_{B,0} \otimes \complex = \dim H^{d-2}(\mathcal{Y}_0(\complex), \complex)^{(prim)}_{\underline{a}} = n.
	\]
	For the second part we observe that
	\[
		\sum_{i = 0}^{d - 1} [\beta_j - a_i] = \binom{d}{2} + \sum_{i = 0}^{d-1} [\beta_j - \alpha_i] - \sum_{i = 0}^{d-1} [\beta_j - \beta_i].
	\]
	This proves the claim for $t = 0$ by paragraph \ref{fermat_para}. Since the Hodge filtration is given by vector bundles (which have locally constant dimension), this also proves the claim for general $t$.
\end{proof}

\begin{example}
	For $d = 3$, we have the isomorphisms
	\[
	H^{1,0}(\mathcal{Y}_0(\complex)) = R \left(\frac{\Omega}{F}\right)\complex \qquad 
	H^{0,1}(\mathcal{Y}_0(\complex)) = R \left( \frac{x_0 x_1 x_2 \Omega}{F^2} \right) \complex
	\]
\end{example}

\begin{example}
	For $d = 5$ the dimension of $H^{3}(\mathcal{Y}_0(\complex))$ is 204 and we have the isomorphisms
	\[
	H^{3 - l, l}(\mathcal{Y}_0(\complex))^{H_0} = R \left( \frac{(x_0 x_1 x_2 x_3 x_4)^{l} \Omega}{F^{l+1}} \right) \complex .  
	\] 
	Let
	$\underline{a} = (0,-1,-1,-1,-2)$. Then there are isomorphisms
	\begin{align*}
		H^{2,1}(\mathcal{Y}_{0}(\complex))_{\underline{a}} & = R \left( \frac{x_1 x_2 x_3 x_4^2 \Omega}{F^2} \right) \complex \\ H^{1,2}(\mathcal{Y}_{0}(\complex))_{\underline{a}} &= R \left( \frac{x_0 x_1^2 x_2^2 x_3^2 x_4^3 \Omega}{F^3} \right) \complex.
	\end{align*}
\end{example}

\begin{lemma}
	Suppose $(\alpha_1, \dots, \alpha_n ; \beta_1, \dots, \beta_n)$ is a hypergeometric parameter of dimension $n$ modulo $d$ such that for every $s \in (\ints/d\ints)^\times$, the following $n$ integers are pairwise distinct
	\[
	\sum_{i = 1}^n [-s\alpha_i + s\beta_j] - \sum_{i \neq j} [s\beta_j - s\beta_i] \qquad j = 1,\dots, n.
	\] 
	Let $\lambda < \mathcal{O}_d$ be a prime above $\ell$ and let $\iota : \mathcal{O}_{d,\lambda} \to \overline{\rationals}_\ell$ be an embedding. If $F \subset \overline{\rationals}_\ell$ is a finite extension of $\rationals_\ell(\mu_d)$ and $t \in F \setminus \mu_d$, then for every embedding $\tau : F \hookrightarrow \overline{\rationals}_\ell$, the $\tau$-Hodge--Tate weights of 
	\[
		M(\alpha_1, \dots, \alpha_n; \beta_1, \dots, \beta_n)_{\lambda, t} \otimes_{\iota} \overline{\rationals}_\ell
	\]
	are pairwise distinct.
\end{lemma}

\begin{proof}
	Let $\underline{a} \in (\ints/d\ints)^d$ be the vector from definition \ref{motive_def}. Let $h \geq 0$ be an integer and $\tau : F \hookrightarrow \overline{\rationals}_\ell$ a field embedding.
	By the Hodge--Tate comparison theorem \cite{faltings88},
	$h$ is a $\tau$-Hodge--Tate weight of 
	\[
		M(\alpha_i;\beta_j)_{\lambda, t} \otimes_{\iota} \overline{\rationals}_\ell = H^{d-2}(\mathcal{Y}_{t, \overline{F}}, \mathcal{O}_{d,\lambda} \otimes_{\iota} \overline{\rationals}_\ell)_{\underline{a}}
	\] 
	of multiplicity
	\[
	\dim H^{d - 2 - h}(\mathcal{Y}_{\tau(t), \overline{\rationals}_\ell}, \Omega^h)_{\underline{a}}.
	\]
	Let $s \in (\ints/d\ints)^\times$ be the unique element such that 
	$\iota^{-1}(\tau(\iota(e^{2 \pi i /d}))) = e^{2 \pi i s/d}$. We have a commutative diagram
	\[
	\begin{tikzcd}
		\mathcal{Y}_{\tau(t), \overline{\rationals}_\ell} \arrow{d} \arrow{r}{\tau^*} & \mathcal{Y}_{t,\overline{\rationals}_\ell} \arrow{d} \\
		\Spec \mathcal{O}_d \arrow{r}{\sigma_s} & \Spec \mathcal{O}_d 
	\end{tikzcd}
	\]
	where $\sigma_s$ is the automorphism of $\mathcal{O}_d$ such that
	$\sigma_s(e^{2 \pi i /d}) = e^{2 \pi i s/d}$. Thus, pullback by $\tau$ induces an isomorphism
	\[
	H^{d - 2 - h}(\mathcal{Y}_{t, \overline{\rationals}_\ell}, \Omega^h)_{s^{-1}\underline{a}} \cong H^{d - 2 - h}(\mathcal{Y}_{\tau(t)}, \Omega^h)_{\underline{a}}
	\]
	and the claim follows from lemma \ref{hodge_numbers_prop}.
\end{proof}

\subsection{The Picard--Fuchs Equation}
In the previous section we studied the de Rham cohomology of $\mathcal{Y}_0$ and computed its Hodge filtration. We are now ready to extend this to a study of the whole local system $R^{d-2} f_* (\underline{\complex})$ on $\Spec A_d(\complex) = \complex \setminus \mu_d$ via the differential equations satisfied by the periods of $\mathcal{Y}_t$.

\begin{num}
	Let $t_0 \in \complex \setminus \mu_d$ and $U \subset \complex \setminus \mu_d$ a connected and simply connected open neighbourhood of $t_0$. For $\gamma_0 \in H_{d-2}(\mathcal{Y}_{t_0}(\complex), \mathcal{O}_d)$ and $t \in U$, there exists a well-defined homology class $\gamma_t \in H_{d-2}(\mathcal{Y}_t(\complex), \mathcal{O}_d)$ such that 
	$\gamma_0$ and $\gamma_t$ pull back to the same class in $H_{d-2}(f^{-1}(U), \mathcal{O}_d)$.

	Moreover, $\tau(\gamma_0)$ and $\tau(\gamma_t)$ are the pull-back of a single class lying in $H_{d-1}(\mathbb{P}^{d-1}(\complex) \setminus f^{-1}(U), \mathcal{O}_d)$. In particular, one can integrate a $(d-1)$-form $\omega$ on $\mathbb{P}^{d-1}(\complex) \setminus \mathcal{Y}_t(\complex)$ over $\tau(\gamma_0)$ and if $d \omega = 0$, we have
	\[
		\int_{\tau(\gamma_0)} \omega = \int_{\tau(\gamma_t)} \omega.
	\]
\end{num}

\begin{definition}
	For a complex number $z$ and a non-negative integer $j$, we define the Pochhammer 
	symbol
	\[
		(z)_j = \begin{cases}
			1 & j = 0 \\
			(z + j - 1) (z)_{j-1} & j > 0
		\end{cases}.
	\]
\end{definition}

\begin{lemma}[Kloosterman] \label{kloosterman_lemma}
	Let $e_i$ be non-negative integers such that $\sum_{i} e_i = d m$. Let $q_i, r_i$ be non-negative integers such that $e_i = d q_i + r_i$. If there exists $i \in \{0, \dots, d-1\}$ such that $r_i = d - 1$, then 
	\[
	R \left( \frac{x^{\underline{e}}\Omega}{F^{m+1}} \right) = 0.
	\]
	In general, we have
	\[
	R \left(\frac{x^{\underline{e}}\Omega}{F^{m+1}} \right) = R \left( \frac{\prod_{i=0}^{{d-1}} ((r_i + 1)/d)_{q_i} }{(l + 1)_{m - l}} \frac{x^{\underline{r}}\Omega}{F^{l+1}} \right), 
	\]
	where $d l = \sum_{i} r_i$.
\end{lemma}

\begin{proof}
	When $r_i < d$ this is precisely the case of \cite[Lemma 5.1]{kloosterman_fermat} with $w_i = 1$ and $w = d$. The proof also works for general $r_i$ without change.
\end{proof}

\begin{lemma} \label{hypergeometric_series_lemma}
	Let $(\alpha_1, \dots, \alpha_n ; \beta_1, \dots, \beta_n)$ be a hypergeometric parameter modulo $d$ and define $\underline{a} \in (\ints/d\ints)^d$ as in definition \ref{motive_def}.
	Let $r_i = [-1 - \beta_1 - a_i]$, $m = \frac{1}{d} \sum_{i = 0}^{d-1} r_i$ and
	\[
		\omega_t := \frac{x^{\underline{r}} \Omega}{F_t^{m+1}}.
	\]
	For $\gamma_0 \in H_{d-2}(\mathcal{Y}_0(\complex), \mathcal{O}_d)$ there exists $\varepsilon > 0$ such that for all $|t| < \varepsilon$, we have
	\[
	\int_{\tau(\gamma_t)} \omega_t = \sum_{j = 1}^{n} \binom{m + [\beta_1 - \beta_j]}{m} (td)^{[\beta_1- \beta_j]} G_j(t^d) \int_{\tau(\gamma_0)} \frac{x^{\underline{r} + [\beta_1 - \beta_j]}\Omega}{F^{m + [\beta_1 - \beta_j] + 1}},
	\]
	where
	\[
	G_j(z) = 
		\sum_{k \geq 0} \prod_{i = 1}^{n} \frac{((d + [\alpha_i - \beta_1] - [\beta_j - \beta_1])/d)_k }{((d + [\beta_i - \beta_1] - [\beta_j - \beta_1])/d)_k} z^k
	\]
\end{lemma}

\begin{proof}
	Choose a $(d-1)$-cycle $\sigma$ on $\mathbb{P}^{d-1}(\complex) \setminus f^{-1}(U)$ representing $\tau(\gamma_0)$. If $\varepsilon < 1/M$, where $M$ is the minimum of $|F|$ on $\sigma$, then the following power series expansion converges on $\sigma$ for $|t| < \varepsilon$
	\[
		\omega_t = \sum_{k \geq 0} \binom{k + m}{m} d^k t^k \frac{x^{\underline{r} + k} \Omega}{F^{m + k + 1}}.
	\]
	Lemma \ref{kloosterman_lemma} implies that 
	\[
	R \left( \frac{x^{\underline{r} + k} \Omega}{F^{m + k + 1}} \right) = 0
	\]
	unless $k \equiv \beta_1 - \beta_j \pmod{d}$ for some $j \in \{1, \dots, n\}$. 
	Thus, the formula in lemma \ref{kloosterman_lemma} and the fact that $\tau(\gamma_t)$ and $\tau(\gamma_0)$ are homologous imply that
	\[
		\int_{\tau(\gamma_t)} \omega_t = \sum_{j = 1}^{n} \binom{m + [\beta_1 -\beta_j]}{m} (td)^{[\beta_1-\beta_j]} G_j(t^d)  \int_{\tau(\gamma_0)} \frac{ x^{\underline{r} + [\beta_1-\beta_j]} \Omega}{F^{m + [\beta_1-\beta_j] + 1}},
	\]
	where
	\begin{align*}
		G_j(z) & = \sum_{k \geq 0} \frac{([\beta_1-\beta_j] + 1 + dk)_{m}}{([\beta_1-\beta_j] + 1)_{m}} d^{dk} \frac{\prod_{i = 0}^{d-1} ((r_i + [\beta_1-\beta_j] + 1)/d)_{k}}{(m + [\beta_1-\beta_j] + 1)_{dk}} z^k \\
		&= \sum_{k \geq 0} \frac{\prod_{i = 0}^{d-1} ((r_i + [\beta_1-\beta_j] + 1)/d)_{k}}{d^{-dk} ([\beta_1-\beta_j] + 1)_{dk}} z^k \\
		&= \sum_{k \geq 0} \prod_{i = 1}^{d-1} \frac{((r_i + [\beta_1-\beta_j] + 1)/d)_{k}}{((i + [\beta_1-\beta_j] + 1)/d)_k} z^k.
	\end{align*}
	Cancelling factors in numerator and denominator we obtain the equality
	\begin{align*}
		\prod_{i = 1}^{d-1} \frac{((r_i + [\beta_1-\beta_j] + 1)/d)_{k}}{((i + [\beta_1-\beta_j] + 1)/d)_k} &= \prod_{i = 1}^n \frac{ (([-1 -\beta_1 + \alpha_i] + [\beta_1 - \beta_j] +1)/d)_k}{(([-1 -\beta_1 + \beta_i] + [\beta_1 - \beta_j] + 1)/d)_k} \\
		 &= \prod_{i = 1}^n \frac{ ((d + [-\beta_1 + \alpha_i] - [\beta_j - \beta_1] )/d)_k}{((d + [-\beta_1 + \beta_i] - [\beta_j - \beta_1])/d)_k}
	\end{align*}
	and hence the claimed formula for $G_j$.
\end{proof}

\begin{lemma} \label{diffeq_lemma}
	Let $(\alpha_1, \dots, \alpha_n ; \beta_1, \dots, \beta_n)$ be a hypergeometric parameter modulo $d$ and define $\underline{a}$ as in definition \ref{motive_def}.
	Let $r_i = [-1 - \beta_1 - a_i]$, $m = \frac{1}{d} \sum_{i = 0}^{d-1} r_i$ and
	\[
	\omega_t := \frac{x^{\underline{r}} \Omega}{F_t^{m+1}}.
	\]
	Let $t_0 \in \complex \setminus \mu_{d}$ and $\gamma_0 \in H_{d-2}(\mathcal{Y}_{t_0}(\complex), \ints)$. For $|t - t_0|$ sufficiently small, the function
	\[
		u(t) = \int_{\gamma_t} R(\omega_t)
	\]
	is holomorphic and is annihilated by the differential operator
	\[
		D := \prod_{i = 1}^n (\theta + [\beta_i - \beta_1] - d)- t^d \prod_{i = 1}^n (\theta + [\alpha_i - \beta_1]),
	\]
	where $\theta = t \frac{d}{dt}$.
\end{lemma}

\begin{proof}
	We can differentiate under the integral sign in the formula
	\[
		u(t) = \int_{\gamma_t} R(\omega_t) = \int_{\tau(\gamma_t)} \omega_t = \int_{\tau(\gamma_0)} \omega_t
	\]
	to prove that $u(t)$ is holomorphic.
	
	If $t_0 = 0$, the differential equation is satisfied on a neighbourhood of $t_0$ by \cite[(2.9)]{beukers_heckman} since
	\[
		G_j(z) = {}_{n}F_{n-1}\left( \begin{array}{c c c} 1 + \alpha'_1/d - \beta'_j/d, & \dots, &  1 + \alpha'_n/d - \beta'_j/d \\
			1 + \beta'_1/d - \beta'_j/d, &  \widehat{\dots}, & 1 + \beta'_n/d - \beta'_j/d \end{array} ; z \right),
	\]
	where $\beta_i' = [\beta_i - \beta_1]$ and $\alpha_i' = [\alpha_i - \beta_1]$.
	For arbitrary $t_0$ we can find a simply-connected neighbourhood $U$ of $t_0$ which contains $0$ and extend $\gamma_t$ to $U$. Then 
	$D u$ is a holomorphic function on $U$ which vanishes on a non-empty open set. Consequently, $D u = 0$ on $U$.
\end{proof}

\begin{corollary} \label{local_system_cor}
	In the notation of the previous lemma, the following map is an isomorphism of local systems
	\[
		(R^{d-2} f_*(\underline{\mathcal{\complex}}))_{\underline{a}} \to \operatorname{Sol}(D) : \sigma_t \mapsto \int_{\mathcal{Y}_t(\complex)} R(\omega_t) \wedge \sigma_t
	\]
	on $\complex \setminus \mu_d$, where $\operatorname{Sol}(D)$ denotes the local system of solutions of $D u = 0$.
\end{corollary}

\begin{proof}
	By lemma \ref{diffeq_lemma}, the map is well-defined. Both source and target have rank $n$. For the source it follows from the computation at $t = 0$ in paragraph \ref{fermat_para} and for the target we refer to \cite{beukers_heckman}.
	
	Lemma \ref{hypergeometric_series_lemma} implies that over a small neighbourhood of $t = 0$, the image of the map in question contains the $n$ independent hypergeometric series $G_j$. Thus, the map of local systems is surjective, hence an isomorphism.
\end{proof}

\subsection{The Geometric Monodromy} Now that we have determined the Picard--Fuchs equation we can describe the image of the geometric monodromy representation with the results of \cite{beukers_heckman}.

\begin{num}
	Let $d > 2$ be an integer. We identify $\pi_1(\mathbb{P}^1(\complex) \setminus \{0, 1, \infty\}, 1/2^d)$ with the group 
	\[
	\Gamma_{0,1,\infty} := \langle \gamma_0, \gamma_1, \gamma_\infty \mid \gamma_0 \gamma_1 \gamma_\infty \rangle.
	\] 
	Let $[d] : \mathbb{P}^1 \to \mathbb{P}^1$ denote the map $z \mapsto z^d$. It induces an injective group homomorphism
	\[
		[d]_* : \pi_1(\mathbb{P}^1(\complex) \setminus (\mu_d \cup \{0, \infty\}), 1/2) \to \pi_1(\mathbb{P}^1(\complex) \setminus \{0, 1, \infty\}, 1/2^d).
	\]
	The inclusion $\mathbb{P}^1(\complex) \setminus (\mu_d \cup \{0, \infty\}) \to \Spec A_d(\complex)$ induces a surjective group homomorphism 
	\[
		\eta : \pi_1(\mathbb{P}^1(\complex) \setminus (\mu_d \cup \{0, \infty\}), 1/2) \twoheadrightarrow \pi_1(\Spec A_d(\complex), 1/2).
	\]
\end{num}

\begin{definition} \label{hypergeometric_group_def}
	Let $(\alpha_1, \dots, \alpha_n ; \beta_1, \dots, \beta_n)$ be a hypergeometric parameter of dimension $n$ modulo $d$.
	Define coefficients $A_i \in \mathcal{O}_d$ and $B_i \in \mathcal{O}_d$ by the equations
	\[
		\prod_{j = 1}^{R} (X - e^{2 \pi i \alpha_j/d}) = \sum_{j = 0}^{R} A_j X^j \qquad 
		\prod_{j = 1}^{R} (X - e^{2 \pi i \beta_j/d}) = \sum_{j = 0}^{R} B_j X^j.
	\]
	We define the hypergeometric monodromy representation 
	\[
		HG(\alpha_i; \beta_i) : \Gamma_{0,1,\infty} \to \GL_n(\mathcal{O}_d)
	\]
	by
	$\gamma_\infty \mapsto A$, $\gamma_0 \mapsto B^{-1}$ and $\gamma_1 \mapsto A^{-1} B$, where
	\[
		A = \begin{bmatrix}
			0 & 0 & \dots & 0 & - A_0 \\
			1 & 0 & \dots & 0 & - A_1 \\
			0 & 1 & \dots & 0 & - A_2 \\
			 &    & \ddots &  & \\
			0 & 0 & \dots & 1 & - A_{n-1}
		\end{bmatrix}, \qquad
			B = \begin{bmatrix}
		0 & 0 & \dots & 0 & - B_0 \\
		1 & 0 & \dots & 0 & - B_1 \\
		0 & 1 & \dots & 0 & - B_2 \\
		&    & \ddots &  & \\
		0 & 0 & \dots & 1 & - B_{n-1}
	\end{bmatrix}.
	\]
\end{definition}

\begin{lemma} \label{transvection_lemma}
	The matrix $u = HG(\alpha_i ; \beta_i)(\gamma_1)$ is a pseudoreflection, i.e. 
	$u - \id$ has rank one. If $d$ is odd, then $\det u = 1$ and if $d$ is even, then $\det u = - 1$.
\end{lemma}

\begin{proof}
	By \cite[Prop 2.10]{beukers_heckman}, we know that $u - \id$ has rank one and 
	$\det u = e^{2 \pi i \gamma/d}$, where
	\[
		\gamma = \sum_{i = 1}^n (\beta_i - \beta_1) - \sum_{i = 1}^n (\alpha_i - \beta_1) = \sum_{i = 1}^n \beta_i - \sum_{i = 1}^n \alpha_i.
	\]
	Hence, the definition of hypergeometric parameters implies $\gamma = \binom{d}{2}$
	and the claim follows. 
\end{proof}

\begin{lemma} \label{monodromy_group_lemma}
	Let $(\alpha_1, \dots, \alpha_n ; \beta_1, \dots, \beta_n)$ be a hypergeometric parameter of dimension $n$ modulo $d$.
	There exists a basis of $M(\alpha_i;\beta_j)_{B,1/2} \otimes_{\mathcal O} \complex$ in which the action of 
	$\pi_1(\Spec A_d(\complex), 1/2)$ on $M(\alpha_i;\beta_j)_{B,1/2} \otimes_{\mathcal O} \complex$ is given by the unique group homomorphism 
	\[
		\rho_{1/2} : \pi_1(\Spec A_d (\complex), 1/2) \to \GL_n(\complex)
	\] 
	making the following diagram commute.
	\[
		\begin{tikzcd}
			\pi_1(\mathbb{P}^1(\complex) \setminus (\mu_d \cup \{0, \infty\}), 1/2) \arrow{r}{[d]_{*}} \arrow{d}{\eta} & \pi_1(\mathbb{P}^1(\complex) \setminus \{0, 1, \infty\}, 1/2^d) \arrow{d}{HG(\alpha_i; \beta_i)} \\
			\pi_1(\Spec A_d(\complex), 1/2) \arrow{r}{\rho_{1/2}} & \GL_n(\complex)
		\end{tikzcd}
	\]
\end{lemma}

\begin{proof}
	Let $D$ be the differential operator from lemma \ref{diffeq_lemma}. 
	Observe that $d^{-n} D$ is the pullback under $[d]$ of the hypergeometric differential operator studied in \cite{beukers_heckman}. 
	Hence, the claim follows from lemma \ref{local_system_cor} and the Theorem of Levelt \cite[Thm 3.5]{beukers_heckman}.
\end{proof}

\begin{remark} \label{sl_remark}
	The homomorphism $\rho_{1/2}$ actually maps into $\{g \in \GL_n(\mathcal{O}_d) : \det(g)^2 = 1\} \subset \GL_n(\complex)$ since $\det(A^d) = \det(B^d) = 1$ and $\det(A^{-1}B) \in \{\pm 1\}$. If $d$ is odd, then $\rho_{1/2}$ maps into $\SL_n(\mathcal{O}_d)$.
\end{remark}

\begin{lemma} \label{constant_det_lemma}
	If $d$ is even, let $\widetilde{A}_d = A_d[\sqrt{1 - t^d}]$ and if $d$ is odd, let $\widetilde{A}_d = A_d$. 	Let $(\alpha_1, \dots, \alpha_n ; \beta_1, \dots, \beta_n)$ be a hypergeometric parameter of dimension $n$ modulo $d$. Denote by $\mathcal{M}_{\lambda}$ the $\mathcal{O}_{d,\lambda}$-etale local system on $\Spec A_d$ corresponding to $M(\alpha_1, \dots, \alpha_n ; \beta_1, \dots, \beta_n)_\lambda$.
	For any map $\iota : \mathcal{O}_{d} \to \complex$, the local system
	\[
	\det (\mathcal{M}_\lambda) \rvert_{\Spec (\widetilde{A}_d \otimes_{\iota} \complex)}
	\]
	is constant.
\end{lemma}

\begin{proof}
	After replacing $(\alpha_i ; \beta_j)$ by $(s \alpha_i ; s \beta_j)$, we may assume that $\iota$ is the inclusion. Thus, the geometric monodromy action on $\mathcal{M}_\lambda$ is given by $\rho_{1/2}$.
	When $d$ is odd, the claim follows from the previous remark. So we assume that $d$ is even.
	The image of $\pi_1(\Spec \widetilde{A}_d(\complex))$ in $\pi_1(\Spec A_d(\complex), 1/2)$ is the kernel of the homomorphism $\pi_1(\Spec A_d(\complex), 1/2) \to \{\pm 1\}$ which sends any simple loop around a $d$th root of unity to $-1$. By lemma \ref{transvection_lemma}, this homomorphism is equal to $\det \rho_{1/2}$. Thus, the restriction of $\det(\mathcal{M}_\lambda)$ to $\Spec (\widetilde{A}_d \otimes_{\iota} \complex)$ is constant.
\end{proof}

\begin{prop} \label{unipotent_monodromy_prop}
	Let $(\alpha_1, \dots, \alpha_n ; \beta_1, \dots, \beta_n)$ be a hypergeometric parameter of dimension $n$ modulo $d$ and 
\[
M = M(\alpha_1, \dots, \alpha_n ; \beta_1, \dots, \beta_n).
\] 
For $t \in \complex \setminus \mu_d$ let $\gamma_\infty \in \pi_1(\complex \setminus \mu_d, t)$ be a simple loop based at $t$ going around $\infty$. Then $\gamma_{\infty}$ acts on $M_{B,t}$ as a unipotent operator with Jordan blocks $J_\alpha$ of size $\#\{i : \alpha_i = \alpha\}$ for $\alpha \in \{\alpha_i\}$.
Similarly, let $\gamma_1$ denote a simple loop around $1$. Then $\gamma_{1}$ acts as a unipotent operator on $M_{B,t}$ and
the rank of $\gamma_{1} - \id$ is $1$ if $d$ is odd and $0$ if $d$ is even.
\end{prop}

\begin{proof}
	The first part follows directly from lemma \ref{monodromy_group_lemma} and the definition of the matrix $A$. The second part is implied by \ref{transvection_lemma}.
\end{proof}

\begin{prop} \label{big_monodromy_prop}
	Let $(\alpha_1, \dots, \alpha_n ; \beta_1, \dots, \beta_n)$ be a hypergeometric parameter of dimension $n$ modulo $d$ and $t \in \complex \setminus \mu_d$. If the following hypotheses are satisfied 
	\begin{itemize}
		\item $|\{\alpha_1, \dots, \alpha_n\}| < n$,
		\item $\{\beta_1, \dots, \beta_n\}$ does not form an arithmetic progression,
		\item there is no non-zero $s \in \ints/d\ints$ such that
		\[
		\{\alpha_i + s\} = \{\alpha_i\} \quad \wedge \quad \{\beta_i + s\} = \{\beta_i\},
		\]
		\item there is no $s \in \ints/d\ints$ such that
		\[
		\{-\alpha_i - s\} = \{\alpha_i + s\} \quad \wedge \quad \{-\beta_i -s\} = \{\beta_i + s\},
		\]
	\end{itemize}
	then the identity component of the Zariski closure of the image of the geometric monodromy representation
	\[
	\pi_1(\complex \setminus \mu_d, t) \to \GL(M(\alpha_1, \dots, \alpha_n ; \beta_1, \dots, \beta_n)_{B,t} \otimes_{\mathcal{O}_d} \complex)
	\]
	contains $\SL_n(\complex)$.
\end{prop}

\begin{proof}
	It follows from \cite[Theorem 6.5]{beukers_heckman} that the Zariski closure of the image of $HG(\alpha_i, \beta_j)$ contains $\SL_n(\complex)$. Namely, the conditions are satisfied by \cite[Theorem 5.3 and Theorem 5.8]{beukers_heckman} and the fact that any scalar shift of the hypergeometric group with parameters $(e^{2 \pi i \alpha_i/d} ; e^{2 \pi i \beta_j/d})$ contains an element of infinite order since $\{\alpha_i\}$ contains a repeated element.
	
	Using lemma \ref{monodromy_group_lemma} we see that the identity component of the Zariski closure of the image of 
	\[
	\pi_1(\complex \setminus \mu_d, t) \to \GL(M(\alpha_i ; \beta_i)_{B,t} \otimes_{\mathcal{O}_d} \complex)
	\]
	is a Zariski-closed finite index subgroup of $\SL_n(\complex)$. But $\SL_n(\complex)$ is connected, hence any such group is equal to $\SL_n(\complex)$ itself.
\end{proof}

\begin{prop} \label{big_monodromy_finite_prop}
	Let $(\alpha_1, \dots, \alpha_n ; \beta_1, \dots, \beta_n)$ be a hypergeometric parameter of dimension $n$ modulo $d$ and $t \in \complex \setminus \mu_d$ satisfying the hypotheses of the previous proposition. Moreover, let $U,V < (\ints/ d\ints)^\times$ be subgroups such that $\{1\} = U \cap V$, $UV = (\ints/d\ints)^\times$ and for $s \in (\ints/d\ints)^\times$ the equalities
	\begin{align*}
		\{ s\alpha_i - s\alpha_j : 1 \leq i,j \leq n \} &= \{\alpha_i - \alpha_j : 1 \leq i,j \leq n\} \\ 
		\{ s\beta_i - s\beta_j : 1 \leq i,j \leq n \} &= \{\beta_i - \beta_j : 1 \leq i,j \leq n\}
	\end{align*}
	imply that $s \in U$. Let $\lambda < \mathcal{O}_d$ be a sufficiently large prime of residue characteristic $\ell$ such that $\ell \in V$. Then the image of
	\[
\pi_1(\Spec A_d(\complex), t) \to \GL(M(\alpha_i ; \beta_j)_{B,t} \otimes_{\mathcal{O}_d} \mathcal{O}_d/\lambda) 
\]
	contains $\SL_n(\mathcal{O}_d/\lambda)$.
\end{prop}

\begin{proof}
	Let $F \subset \rationals(\mu_d)$ be the fixed field of $U < (\ints/d\ints)^\times = \Gal(\rationals(\mu_d)/\rationals)$ and $K \subset \rationals(\mu_d)$ the fixed field of $V$.
	Let $X = \Spec \mathcal{O}_K[1/(n |V| - 1)!]$ and let $G$ be the group scheme $(\Res_{ F / \rationals } \SL_n)_{X}$ and consider
	the homomorphism 
	\begin{align*}
		\Gamma' &\xrightarrow{\Phi} G(\mathcal{O}_K[1/(n|V| - 1)!]) = \SL_n(\mathcal{O}_d[1/(n|V| - 1)!]) \\
		\gamma &\mapsto HG(\alpha_i; \beta_j)(\gamma),
	\end{align*}
	where $\Gamma' < \Gamma_{0,1, \infty}$ is the kernel of $\det HG(\alpha_i; \beta_j)$.
	Choose an embdedding
	$G \hookrightarrow \GL_{n |V|,X}$. 
	
	Let $W \subset \mathfrak{g} := \operatorname{Lie}(G)$ be the $\mathcal{O}_K[1/(n|V| - 1)!]$-module generated by elements $\log \gamma$, where $\gamma$ is a unipotent element in the image of $\Phi \circ [d]_*$. Equip $\mathfrak{g}$ with the conjugation action of $\Gamma_{0,1, \infty}$ via $\Phi$. Then $W$ is a $\Gamma_{0,1,\infty}$-equivariant subspace.
	
	We have $\mathfrak{g} \otimes_{\mathcal O_K} \complex = \prod_{v \in V} \mathfrak{sl}_{n}$. Since the image of $HG(\alpha_i ; \beta_j)$ is Zariski-dense and $\mathfrak{sl}_n$ is an irreducible $\SL_n(\complex)$-representation, we see that for each $v \in V$, the $\mathfrak{sl}_n$-summand of $\mathfrak{g}$ is an irreducible $\Gamma_{0,1,\infty}$-representation.
	Now the assumption on the sets $\{\alpha_i - \alpha_j\}$ and $\{\beta_i - \beta_j\}$ imply that $\mathfrak{g} \otimes \complex$ is a multiplicity-free semisimple $\Gamma_{0,1,\infty}$-module.
	
	The element $\log \gamma_\infty^d \in W$ has a non-zero projection in each irreducible summand of $\mathfrak{g} \otimes \complex$. Thus, $W \otimes \complex = \mathfrak{g} \otimes \complex$ and there exists a constant $C$ such that $W[1/C] = \mathfrak{g}[1/C]$. Let $\lambda' < \mathcal{O}_K$ be the prime below $\lambda$. Since $\ell \in V$, we have $\mathcal{O}_K/\lambda' \cong \mathbb{F}_\ell$. We apply \cite[Theorem 12.4.1]{katz_gkm} with $R = \mathcal{O}_K/\lambda'$ to see that $G(R) = \SL_n(\mathcal{O}_d/\lambda' \mathcal{O}_d)$ is generated by the unipotent elements in the image of $\pi_1(\complex \setminus \mu_d, t) \to \GL_n(\mathcal{O}_d/\lambda'\mathcal{O}_d)$. In particular, the image of $\pi_1(\complex \setminus \mu_d, t) \to \GL_n(\mathcal{O}_d/\lambda)$ contains $\SL_n(\mathcal{O}_d/\lambda)$.
\end{proof}

\begin{remark}
	If we take $V = \{1\}$, $U = (\ints/d\ints)^\times$ and $\ell \equiv 1 \pmod{d}$, then the conditions in the proposition are satisfied and $\mathcal{O}_d/\lambda \cong \mathbb{F}_\ell$.
\end{remark}

\begin{corollary}
	If $\lambda_1, \lambda_2 < \mathcal{O}_d$ are sufficiently large primes above $\ell_1 \neq \ell_2$ such that $\ell_1, \ell_2 \in V$, then the image of 
	\[
		\pi_1(\Spec A_d(\complex), t) \to \GL(M(\alpha_i ; \beta_j)_{B,t} \otimes_{\mathcal{O}_d} \mathcal{O}_d/\lambda_1\lambda_2) 
	\]
	contains $\SL(M(\alpha_i ; \beta_j)_{B,t} \otimes_{\mathcal{O}_d} \mathcal{O}_d/\lambda_1\lambda_2)$.
\end{corollary}

\begin{proof}
	By the proposition we already know that each of the projections
	$\ker \rho_{1/2} \to \SL_n(\mathcal{O}_d/\lambda_i)$ is surjective. But for different primes $\ell_1, \ell_2 > 3$, there are no proper normal subgroups $N_i < \SL_n(\mathcal{O}_d/\lambda_i)$ such that there exists an isomorphism $\SL_n(\mathcal{O}_d/\lambda_1)/N_1 \cong \SL_n(\mathcal{O}_d/\lambda_2)/N_2$ because
	$\PSL_n(\mathcal{O}_d/\lambda_1)$ and $\PSL_n(\mathcal{O}_d/\lambda_2)$ are distinct simple groups. Thus, the claim follows from Goursat's lemma.
\end{proof}

\subsection{The Arithmetic Monodromy} Now that we have studied the action of the fundamental group on $M(\alpha_i, \beta_j)_{B, t}$ we turn to the Galois action on $M(\alpha_i, \beta_j)_{\lambda, t}$. In particular, we will describe its determinant.

\begin{num} \label{jacobi_character_par}
	Let $d$ be a positive integer and $k$ a finite field of cardinality $q$ coprime to $d$ and let $h : \mathcal{O}_d \to k$ be a ring homomorphism.
	Let 
	\[
	t : \{ y \in k : y^d = 1 \} \to \mathcal{O}_d^\times
	\]
	be the unique map such that $h(t(y)) = y$. We define $\tau(y) := t(y^{(q-1)/d})$. Choose a non-trivial additive character $e : k \to \complex^\times$ and define the Gauss sums
	\[
	g(k, a) = -\sum_{x \in k^\times} \tau(x)^{-a} e(x) \in \complex
	\]
	for $a \in \ints/d\ints$. For $\underline{a} \in (\ints/d\ints)^{m + 1}$ such that $\sum a_i = 0$, we also define the Jacobi sum
	\[
		J_{\underline{a}}(k) = (-1)^m \sum_{\substack{x_1, \dots, x_m \in k \\ x_1 + \dots + x_m = -1}} \prod_{i = 1}^m \tau(x_i)^{-a_i} \in \mathcal{O}_d.
	\]
	When $\underline{a} \neq (0, \dots, 0)$, the relation
	\[
		J_{\underline{a}}(k) = q^{-1} g(k, a_0) \cdots g(k, a_m)
			\]
	holds \cite[I, Lemma 7.9]{deligne_milne_ogus_shih}. It is proved in \cite{weil_jacobi_sum}, that 
	\[
		\psi_{\underline{a}} : \mathfrak{p} \mapsto J_{\underline{a}}(k(\mathfrak{p}))
	\] 
	defines an algebraic Hecke character of $\rationals(\mu_d)$, unramified outside $d$. By abuse of notation we will also denote the corresponding character $G_{\rationals(\mu_d)} \to \mathcal{O}_{d,\lambda}^\times$ by $\psi_{\underline{a}}$. 
\end{num}

\begin{lemma} \label{determinant_jacobi_sum_lemma}
	Let $(\alpha_1, \dots, \alpha_n ; \beta_1, \dots, \beta_n)$ be a hypergeometric parameter of dimension $n$ modulo $d$. Then for all primes $\lambda < \mathcal{O}_d$ with residue characteristic different from the characteristic of $k(\lambda)$, the geometric Frobenius $\varphi_\lambda$ acts on
	\[
		\bigwedge^n M(\alpha_1, \dots, \alpha_n ; \beta_1, \dots, \beta_n)_{\lambda, 0}
	\]
	as multiplication by
	\[
		\prod_{i = 1}^n J_{\underline{a} + \beta_i}(k(\lambda)),
	\]
	where $\underline{a}$ is the vector defined in \ref{motive_def}.
\end{lemma}

\begin{proof}
	This follows from \cite[I, Proposition 7.10]{deligne_milne_ogus_shih}.
\end{proof}

\begin{num}
	For $a \in (\ints/d\ints) \setminus \{0\}$ and a divisor $m \mid d$ such that $m a \neq 0$, we define the functions $\delta_a, \epsilon_{m, a} : (\ints/d\ints) \setminus \{0\} \to \ints$ as follows
	\begin{align*}
		\delta_a(x) = \begin{cases}
			1 & x = a \\
			0 & \text{ otherwise }
		\end{cases}
	\end{align*} 
	and
	\[
	\epsilon_{m,a}(x) = \delta_{- m a} + \sum_{0 \leq j \leq m - 1} \delta_{a + j d/m}.
	\]
	We let $E(d)$ denote the set of functions consisting of $\epsilon_{1,a}$ and $\epsilon_{p,a}$, where $p$ is a prime divisor of $d$. It is proved in \cite[Appendix]{deligne_valeurs} that any $f : (\ints/d\ints) \setminus \{0\} \to \rationals$, such that 
	\[
	\langle f \rangle (s) := \frac{1}{d}\sum_{a = 1}^{d-1} f(a) [sa]
	\]
	is independent of $s \in (\ints/d\ints)^\times$, is a linear combination of elements of $E(d)$. For such a function $f$, we define the complex number
	\[
		\Gamma(f) := \frac{1}{(2 \pi i)^{\langle f \rangle(1)}} \prod_{a = 1}^{d-1} \Gamma(a/d)^{f(a)}.
	\]
\end{num}

\begin{lemma} \label{gamma_algebraic_lemma}
	Let $f : (\ints/ d \ints) \setminus \{0\} \to \ints$ be a function such that $w = \langle f \rangle(s)$ is an integer independent of $s \in (\ints/d\ints)^\times$. Then 
	$\Gamma(f)$
	is an algebraic number and generates an abelian extension of $\rationals(\mu_d)$ unnramified outside $2d$. For $\mathfrak{p} < \mathcal{O}_d$ a prime not dividing $2d$ we have
	\[
		F_{\mathfrak{p}} (\Gamma(f)) = \frac{1}{(N\mathfrak{p})^w} \prod_{a = 1}^{d-1} g(k(\mathfrak{p}), a)^{f(a)} \Gamma(f),
	\]
	where $F_{\mathfrak{p}} \in \Gal(\overline{\rationals}/\rationals(\mu_d))^{ab}$ is the geometric Frobenius element attached to $\mathfrak{p}$.
\end{lemma}

\begin{proof}
	See \cite[I, Theorem 7.15]{deligne_milne_ogus_shih}.
\end{proof}

\begin{lemma} \label{gamma_epsilon_lemma}
	For an integer $1 \leq a \leq d/2$, we have
	\[
		\Gamma(\epsilon_{1,a}) = \frac{1}{e^{\pi i a / d} - e^{- \pi i a/d}} \in e^{\pi i a/d} \cdot \rationals(\mu_d)
	\]
	and for a prime divisor $p \mid d$ and $1 \leq a < d/p$, we have
	\[
		\Gamma(\epsilon_{p,a}) = \frac{p^{1/2 - a/d}}{ i^{(p-1)/2} (e^{\pi i ap/d} - e^{-\pi i ap/d}) } \in e^{\pi i (ap/d + (p-1)/4)} p^{1/2 - a/d} \cdot \rationals(\mu_d).
	\]
\end{lemma}

\begin{proof}
	This follows from the two classical identities
	\[
		\Gamma(x) \Gamma(1 - x) = \pi \sin(\pi x)^{-1}
	\]
	and 
	\[
		\prod_{m = 0}^{p-1} \Gamma(x + m/p) = p^{1/2 - x} (2 \pi)^{(p-1)/2} \Gamma(px). \qedhere
	\]
\end{proof}

\begin{lemma} \label{gamma_field_lemma}
	Suppose $f = \sum_{1 \leq a \leq d/2} x_{1,a} \epsilon_{1,a} + \sum_{p \mid d} \sum_{1 \leq a < d/p} x_{p, a} \epsilon_{p,a}$ for some integer coefficients $x_{m,a}$. Then
	\[
		\Gamma(f) \in e^{\pi i y_1}\prod_{p \mid d} p^{y_p} \cdot \rationals(\mu_d),
	\]
	where
	\begin{align*}
	y_p &= \sum_{1 \leq a < d/p} x_{p,a}(1/2 - a/d)  \\
	y_1 &= \sum_{ 1 \leq a \leq d/2} x_{1,a} a/d + \sum_{p \mid d} \sum_{1 \leq a < d/p} x_{p,a} (ap/d + (p-1)/4)
	\end{align*}
\end{lemma}

\begin{proof}
	Note that given two functions $f_1, f_2 : (\ints/d\ints) \setminus \{0\} \to \rationals$ such that $\langle f_i \rangle$ are constant, we have
	$\Gamma(f_1 + f_2) = \Gamma(f_1) \Gamma(f_2)$. The lemma follows from expanding $\Gamma(f)$ as a product of factors of the form $\Gamma(\epsilon_{m,a})$ and collecting terms using lemma \ref{gamma_epsilon_lemma}.
\end{proof}

\begin{corollary} \label{gamma_degree_cor}
	Let $n > 1$ be an integer. Suppose $f = \sum_{1 \leq a \leq d/2} x_{1,a} \epsilon_{1,a} + \sum_{p \mid d} \sum_{1 \leq a < d/p} x_{p, a} \epsilon_{p,a}$ for some integer coefficients $x_{m,a}$.
	Let $y_p$ and $y_1$ be defined as in the lemma and let $b_1$ be the denominator of $y_1$ and 
	\[
		b_p = \begin{cases}
			\text{ denominator of } 2 y_p & 4 \mid d \text{ or } p \equiv 1 \pmod{4} \\
			\text{ denominator of } y_p & \text{ otherwise.} 
		\end{cases}
	\]
	If the $b_p$ are coprime to $n$ and $\varphi(\operatorname{lcm}(2b_1, d))/\varphi(d)$ is coprime to $n$, then
	\[
		[\rationals(\mu_d, \Gamma(f)) : \rationals(\mu_d)]
	\]
	is coprime to $n$.
\end{corollary}

\begin{proof}
	By the lemma, we know that $\rationals(\mu_d, \Gamma(f))$ is contained in the composite field $ E = \rationals(\mu_d, e^{\pi i y_1}) \prod_p \rationals(\mu_d, p^{y_p})$. Note that if $4 \mid d$ or $p \equiv 1 \pmod{4}$, then $p^{1/2} \in \rationals(\mu_d)$.
	Hence, the assumptions imply that 
	\[
	[E : \rationals(\mu_d)] \mid \varphi(\operatorname{lcm}(2b_1,d))/\varphi(d) \prod_{p} b_p
	\]
	is coprime to $n$ and the claim follows from multiplicativity of degrees in field extensions.
\end{proof}

\begin{prop} \label{determinant_nth_power_prop}
	Let $(\alpha_1, \dots, \alpha_n ; \beta_1, \dots, \beta_n)$ be a hypergeometric parameter of dimension $n$ modulo $d$ and $F$ a number field containing the $d$th roots of unity. Let $t \in F \setminus \mu_d$ and $\underline{c} \in (\ints/d\ints)^3$ such that $c_0 + c_1 + c_2 = 0$ and either $c_i \neq 0$ for $i =0,1,2$ or $c_0 = c_1 = c_2 = 0$.
	Suppose the following hypotheses are satisfied:
	\begin{itemize}
		\item If $d$ is even, then $1 - t^d$ is a square.
		\item For all $s \in (\ints/d\ints)^\times$, the following $n$ integers are pairwise distinct
		\[
		\sum_{i} [s \beta_j - s \alpha_i] - \sum_{i} [s \beta_j - s \beta_i] \qquad j = 1, \dots, n.
		\]
		\item The following quantity is independent of $s \in (\ints/d\ints)^\times$
		\[
		\sum_{i,j} [s \beta_j - s \alpha_i] - \sum_{i,j} [s \beta_j - s \beta_i] + n \sum_{i = 0}^2[s c_i].
		\]
		\item For each $s \in (\ints/d\ints)^\times$ we have
		\[
		\sum_{i,j} [s \beta_j - s \alpha_i] = n \sum_i [s \beta_i - s \alpha_i].
		\]
		\item There are integers $x_{\epsilon}$ for $\epsilon \in E(d)$ such that
		\[
		n + \sum_{i,j} \delta_{\beta_j - \alpha_i} - \sum_{i \neq j} \delta_{\beta_j - \beta_i} + n \sum_{i = 0}^2 \delta_{c_i} = \sum_{\epsilon \in E(d)} x_{\epsilon} \epsilon.
		\]
		Let $b_1$ and $b_p$ be the integers defined in corollary \ref{gamma_degree_cor}. The integer
		$\varphi(\operatorname{lcm}(2 b_1, d))/\varphi(d)$ is coprime to $n$ and $b_p$ is coprime to $n$.
	\end{itemize}
	Then there exists a character $\psi : G_{\rationals(\mu_d)} \to \mathcal{O}_{d,\lambda}^\times$, such that 
	\[
		\det(M(\alpha_i ; \beta_j)_{\lambda, t} \otimes \psi) = \chi_{cyc}^{-n(n - 1)/2}.
	\]
\end{prop}

\begin{proof}
	Let $f : (\ints/d\ints) \setminus \{0\} \to \ints$ be defined by
	\[
		f(x) = n + \sum_{i,j} \delta_{\beta_j - \alpha_i}(x) - \sum_{i \neq j} \delta_{\beta_j - \beta_i}(x) + n \sum_{i = 0}^2 \delta_{c_i}(x).
	\]
	Then by assumption, $\langle f \rangle(s)$ is an integer $w$ independent of $s \in (\ints/d\ints)^\times$. 
	Thus, $\Gamma(f)$ is an algebraic number by lemma \ref{gamma_algebraic_lemma} and if $g \in G_{\rationals(\mu_d)}$, we have
	\[
		g(\Gamma(f)) = \chi(g) \Gamma(f),
	\]
	where $\chi : G_{F} \to \rationals(\mu_d)^\times$ is the continuous character such that 
	\[
		\chi(F_{\mathfrak{p}}) = \frac{1}{(N\mathfrak{p})^w} \prod_{a = 1}^{d-1} g(k(\mathfrak{p}), a)^{f(a)}
	\]
	for geometric Frobenius elements $F_{\mathfrak{p}} \in G_{F}^{ab}$.
	On the other hand, $t$ defines an $F$-point of $\widetilde{A}_d$, hence lemmas \ref{constant_det_lemma} and \ref{determinant_jacobi_sum_lemma} show that $F_{\mathfrak{p}}$ acts on $\bigwedge^n(M_{\lambda,t})$ as multiplication by
	\[
		\prod_{i = 1}^n J_{\underline{a} + \beta_i}(k(\mathfrak{p})) = (N \mathfrak{p})^{w - 2n} J_{\underline{c}}(k(\mathfrak{p}))^{-n} \chi(F_{\mathfrak{p}}).
	\]
	Moreover, the third assumption of the proposition implies that
	\[
		w \equiv - \sum_{i \neq j} [\beta_j - \beta_i]  = -n(n-1)/2 \pmod{n}.
	\]
	Corollary \ref{gamma_degree_cor} shows that $\Gamma(f)$ generates an extension of $\rationals(\mu_d)$ of degree coprime to $n$. Equivalently, $\chi = \eta^n$ is an $n$th power. Thus, we may take
	$\psi = \eta^{-1} \psi_{\underline{c}} \chi_{cyc}^{2 - m}$,
	where $m$ is an integer such that $w = mn -n(n-1)/2$.
\end{proof}

\begin{lemma} \label{ordinary_lemma}
	Let $(\alpha_1, \dots, \alpha_n ; \beta_1, \dots, \beta_n)$ be a hypergeometric parameter of dimension $n$ modulo $d$ such that for all $s \in (\ints/d\ints)^\times$, the following $n$ integers are pairwise distinct
	\[
	\sum_{i} [s \beta_j - s \alpha_i] - \sum_{i} [s \beta_j - s \beta_i] \qquad j = 1, \dots, n.
	\]
	Let $\ell \equiv 1 \pmod{d}$ and $\lambda < \mathcal{O}_d$ be a prime above $\ell$.
	Suppose $K/ \rationals_\ell$ is a finite Galois extension and $\eta : \mathcal{O}_{d} \to K$ a morphism of rings inducing $\lambda$.
	If $t \in K$ satisfies $|t| < 1$, then the Galois representation 
	\[
		\rho : G_{K} \to \GL(M(\alpha_i ; \beta_j)_{\lambda,t} \otimes_{\eta} K)
	\]
	is ordinary in the sense of \cite[\S 5.2]{geraghty19}.
\end{lemma}

\begin{proof}
	Note that $\mathcal{Y}_t$ is smooth and proper over $\mathcal{O}_K$. Hence, $\rho$ is crystalline and by \cite[Lemma 2.32]{geraghty19} it remains to compute the valuations of the eigenvalues of the geometric Frobenius acting on 
	\[
		H^{d-2}_{cris}(\mathcal{Y}_{t} / \mathcal{O}_K )_{\underline{a}} \cong H^{d-2}_{cris}(\mathcal{Y}_{0} / \mathcal{O}_{d,\lambda} )_{\underline{a}} \otimes_{\eta} \mathcal{O}_K.
	\]
	One can use \cite[Theorem 2(2)]{katz_messing} to show that 
	these eigenvalues coincide with the Frobenius eigenvalues on $M(\alpha_i ; \beta_j)_{\mathfrak{p}, 0}$ for some auxiliary prime $\mathfrak{p} \nmid \ell$ of $\mathcal{O}_d$. It follows from \cite[I, Proposition 7.10]{deligne_milne_ogus_shih} that the set of eigenvalues is 
	\[
		\{ \eta(J_{\underline{a} + \beta_i}(k(\lambda))) : i = 1,2, \dots, n \}.
	\]
	The formula \cite[(9)]{weil_jacobi_sum} with the congruence $\ell \equiv 1 \pmod{d}$  implies that there exists an element $s \in (\ints/d\ints)^\times$ depending on $\eta$ such that
	\[
		v_\ell( \eta(J_{\underline{a} + \beta_i}(k(\lambda))) ) = \sum_{j = 0}^{d - 1} [s a_j + s \beta_i] = \binom{d}{2} + \sum_{j = 0}^{d-1} [s\beta_i - s\alpha_j] - \sum_{j = 0}^{d-1} [s\beta_i - s\beta_j].
	\]
	By assumption, these are $n$ distinct integers and \cite[Lemma 2.32]{geraghty19} applies.
\end{proof}

\section{Applications to Galois Representations}
	
\subsection{Preliminaries on Weil--Deligne Representations} \label{WD_section}
	
	Let $K / \rationals_p$ be a finite extension and $W_K \subset \Gal(\overline{K}/K)$ the Weil group of $K$. In this section we prove some lemmas about Weil--Deligne representations, which we use in the applications below. For a general introduction to Weil--Deligne representations see \cite[\S 4]{tate_background}.
	
	\begin{definition}
 		A Weil--Deligne representation of $W_K$ is a triple $(\rho, V, N)$ where $V$ is a finite dimensional complex vector space, $\rho : W_K \to \GL(V)$ is a continuous representation and $N \in \End(V)$ is a nilpotent operator such
		that $\rho(g) N = \| g \| N \rho(g)$ for all $g \in W_K$. 
	\end{definition}

	\begin{definition}
		We define the Frobenius semisimplification of a Weil--Deligne representation $(\rho, V, N)$ as 
		\[
		(\rho, V, N)^{F-ss} := (\rho^{ss}, V, N),
		\]
		where $\rho^{ss} : W_K \to \GL(V)$ denotes the semisimplification of $\rho$. The semisimplification of $(\rho, V, N)$ is defined as 
		\[
		(\rho, V, N)^{ss} := (\rho^{ss}, V, 0).
		\]
	\end{definition}

	\begin{definition} \label{generic_def}
		A Weil--Deligne representation $(\rho,V,N)$ is generic if 
		\[
		\Hom((\rho, V, N), (\rho(1), V, N)) = 0.
		\]
	\end{definition}

	\begin{example}
		For any integer $m \geq 1$, the indecomposable Weil--Deligne representation $\Sp(m)$ of rank $m$ defined in \cite[4.1.4]{tate_background} is generic.
	\end{example}

	\begin{lemma}
		Given a nilpotent operator $N$ acting on a finite dimensional vector space $V$, there exists a unique exhaustive filtration $\dots \subset M_i \subset M_{i + 1} \subset \dots$ of $V$ such that
		\begin{itemize}
			\item $N M_i \subset M_{i-2}$
			\item $N^i$ induces an isomorphism $\operatorname{gr}_{i}^M V \to \operatorname{gr}_{-i}^M V$
		\end{itemize}
		$M_\bullet$ is called the monodromy filtration.
	\end{lemma}

	\begin{proof}
		Apply \cite[Prop 1.6.1]{deligne_weil2} to the category of vector spaces.
	\end{proof}

	\begin{corollary} \label{monodromy_filtration_preserved_cor}
		If $f : V \to V$ is a linear automorphism commuting with $N$, then $f M_i = M_i$.
	\end{corollary}

	\begin{proof}
		Consider the filtration $f M_i$.
		We have $N f M_i = f N M_i \subset f M_{i - 2}$. Since $f$ is an automorphism, we also know that
		\[
			f M_{i}/ f M_{i-1} \xrightarrow{N^i} f M_{-i} / f M_{- i - 1}
		\] 
		is an isomorphism. Thus, the uniqueness claim in the lemma implies that $f M_i = M_i$.
	\end{proof}

	\begin{definition}
	A Weil--Deligne representation $(\rho,V,N)$ is pure of weight $w$ if the eigenvalues of any lift of geometric Frobenius acting on $\operatorname{gr}_i^M V$ are Weil numbers of weight $w + i$.
\end{definition}

	\begin{lemma} \label{pure_implies_generic_lemma}
		If a Weil--Deligne representation $(\rho, V, N)$ is pure, then $(\rho, V, N)^{F-ss}$ is generic.
	\end{lemma}

	\begin{proof}
		Suppose $f : V \to V(1)$ is a morphism of Weil--Deligne representations.
		By corollary \ref{monodromy_filtration_preserved_cor}, $f$ preserves the filtration $M_i$ and induces morphisms of $W_K$-representations
		$f_i : gr_i^M V \to gr_i^M V(1)$. By the purity, the eigenvalues of  $\rho^{ss}(\Frob_K)$, which are the same as the eigenvalues of $\rho(\Frob_K)$, on these two vector spaces have different weights. Hence, $f_i = 0$ for all $i$ and $f = 0$. Consequently, $(\rho, V, N)^{F-ss}$ is generic.
	\end{proof}

	\begin{lemma} \label{generic_ineq_lemma}
		Let $\rho : W_K \to \GL(V)$ be a semisimple continuous representation, where $V$ is a finite dimensional complex vector space.
		There exists a generic Weil--Deligne representation of the form $(\rho, V, N)$, which is unique up to isomorphism.
	\end{lemma}

	\begin{proof}
		Decompose $(\rho, V) = \bigoplus_{i = 1}^t (\rho_{i}, V_i)$ into irreducible summands and assume that they are ordered such that for every $i \in \{1, \dots, t\}$, we have $\rho_{i} \not \cong \rho_{j}(k)$ for $j > i$ and $k \geq 1$. We will prove the claim by induction on $t$. If $t = 0$, we can take $N = 0$. Suppose $t > 0$. We treat uniqueness first.
		
		Our ordering implies that $\rho_1$ cannot lie in the image of $N$. Thus,
		there is a direct sum decomposition of Weil--Deligne representations $V = U \oplus S$, where $U = \bigoplus_{k \geq 0} N^{k} \rho_1$. 
		Since $(\rho, V, N)$ is generic, we have
		$\Hom(S, U(1)) = 0$. This is implies that $\rho_1(a) \otimes \Sp(m')$ for $1 \leq a \leq m$ and $m' > m - a$ is not a direct summand of $S$. In particular, $\rho_1(m)$ is not a $W_K$-direct summand of $S$ and $m$ is the maximal integer such that $\rho_1(k) \in \{\rho_i\}$ for $k = 0, \dots, m - 1$. Hence, the isomorphism class of $U \cong \rho_1 \otimes \Sp(m)$ is independent of $N$. By the inductive hypothesis, the isomorphism class of $S$ is independent of $N$, too. We conclude that $(\rho, V, N)$ is uniquely determined.
		
		For the existence we let $m$ be maximal such that $\rho_1(k) \in \{\rho_i\}$ for all $k = 0, \dots, m - 1$. There are indices $i_0, i_1, \dots, i_{m - 1} \in \{1, \dots, t\}$ such that
		$V_{i_j} \cong V_{i_{j-1}}(1)$ and $i_0 = 1$. We define a nilpotent operator 
		\[
			N : \bigoplus_{j = 0}^{m-1} V_{i_j} \to \bigoplus_{j = 0}^{m-1} V_{i_j}(-1)
		\] 
		such that the restrictions induce isomorphisms $N : V_{i_{j - 1}} \to V_{i_{j}}(-1)$ for $j > 0$. This operator defines a Weil--Deligne representation on $U := \bigoplus_{j = 0}^{m-1} V_{i_j}$ which is isomorphic to the generic representation $\rho_1 \otimes \Sp(m)$. 
		Using the inductive hypothesis on $S := \bigoplus_{i \neq i_j} V_i$, we can extend $N$ to all of $V$. We have $\Hom(U, S(1)) = 0$ since $U$ is generated by $\rho_1$ and $\rho_1 \not \cong \rho_i(1)$ for all $i > 1$ and hence for all $i \not \in \{i_0, \dots, i_{m-1}\}$.
		We have $\Hom(S, U(1)) = 0$ since otherwise there would be an index $i \not \in \{i_0, \dots, i_{m-1}\}$ such that $\rho_{i} \cong \rho_1(m)$, contradicting the maximality of $m$.
		Hence $(\rho, V, N)$ is generic.
	\end{proof}

	\begin{remark}
		The previous lemma has the following geometric interpretation: There is an affine variety $X/\complex$ (the Vogan variety introduced in \cite[\S 4]{vogan93}) which parametrises all the possible nilpotent matrices $N$ such that $(\rho, V, N)$ is a Weil--Deligne representation. The centraliser $G$ of $(\rho, V)$ acts on $X$ by conjugation. The space of (framed) deformations of $(\rho, V, N)$ has minimal dimension if and only if $(\rho, V, N)$ lies in an open $G$-orbit. 
		
		A Galois cohomology argument shows that the (framed) deformation space of $(\rho, V, N)$ has dimension 
		\[
			(\dim V)^2 + \dim \Hom( (\rho, V, N), (\rho(1), V, N)),
		\] 
		which is minimal if and only if $(\rho, V, N)$ is generic. Thus the lemma could be restated as saying that $X$ has a unique open $G$-orbit, which follows from
		\cite[Prop 4.5 d)]{vogan93}.
	\end{remark}

	\begin{num} \label{WD_rep_num}
		Let $K/\rationals_p$ be a finite extension and $\ell \neq p$ a prime number. Let $\Frob_K \in G_K$ be a Frobenius element and $t_\ell : I_K \to \ints_\ell$ a choice of tame character, where $I_K < G_K$ is the inertia group. Given an isomorphism $\iota : \overline{\rationals}_\ell \cong \complex$ and a continuous Galois representation $\rho : G_{K} \to \GL_n(\overline{\rationals}_\ell)$ there exists a unique Weil--Deligne representation
		$WD_{\iota}(\rho) = (r, \complex^n, N)$ such that for all $a \in \ints$ and $x \in I_K$, we have
		$\rho^{\iota}(\Frob_K^a x) = r(\Frob_K^a x) \exp(t_\ell(x) N)$. Moreover, the isomorphism class of $WD_{\iota}(\rho)$ does not depend on $\Frob_K$ and $t_\ell$. \cite[4.2.1]{tate_background}
	\end{num}

	\begin{lemma}  \label{purity_lemma}
		Let $(\alpha_1, \dots, \alpha_n ; \beta_1, \dots, \beta_n)$ be a hypergeometric parameter modulo $d$. Let $F$ be a CM field containing the $d$th roots of unity and $t \in F \setminus \mu_d$. Let $\lambda < \mathcal{O}_d$ be a prime of residue characteristic $\ell$ and $v \nmid \ell$ a finite place of $F$. For any isomorphism $\iota : \overline{\rationals}_\ell \cong \complex$, the Weil--Deligne representation
		\[
			WD_{\iota}(M(\alpha_1, \dots, \alpha_n ; \beta_1, \dots, \beta_n)_{\lambda, t}\rvert_{G_{F_v}} \otimes_{\mathcal{O}_{d,\lambda}} \overline{\rationals}_\ell )
		\]
		is pure.
	\end{lemma}

	\begin{proof}
		The representation $M(\alpha_1, \dots, \alpha_n ; \beta_1, \dots, \beta_n)_{\lambda, t}\rvert_{G_{F_v}}$ is a direct summand of the cohomology of a smooth projective hypersurface. Thus, it is pure by \cite[Thm 1.14]{scholze_perfectoid}.  
	\end{proof}

	\begin{num}
		For a smooth irreducible representation $\pi$ of $\GL_n(K)$, we denote by $\operatorname{rec}_K(\pi)$ the Weil--Deligne representation attached to $\pi$ via the local Langlands correspondence of Harris--Taylor \cite{harris_taylor}. 
	\end{num}

	\begin{theorem}[Varma] \label{varma_thm}
		Let $F$ be a CM or totally real field and $\ell$ a prime number. Let $\pi$ be a cuspidal cohomological automorphic representation of $\GL_n(\mathbb{A}_F)$. For a place $v \nmid \ell$ of $F$ we have 
		\[
				WD_{\iota}( r_{\iota}(\pi)\rvert_{G_{F_v}} )^{ss} \cong \operatorname{rec}_{F_v}(\pi_v |\det|^{(1 - n)/2})^{ss}
		\]
		and
		\[
			WD_{\iota}( r_{\iota}(\pi)\rvert_{G_{F_v}} )^{F-ss} \cong \operatorname{rec}_{F_v}(\pi_v |\det|^{(1 - n)/2})
		\]
		holds if and only if $WD_{\iota}(r_{\iota}(\pi)\rvert_{G_{F_v}})^{F-ss}$ is generic.
	\end{theorem}

	\begin{proof}
		The first part follows from the main theorem of \cite{varma_local_global}.
		Since $\pi$ is a cuspidal representation of $\GL_n(\mathbb{A}_F)$, it is generic, and its local factors $\pi_w$ are generic for all finite places $w$.
		By \cite[Lemma 1.1.3]{allen_polarizable}, the Weil--Deligne representation $\operatorname{rec}_{F_v}(\pi_v |\det |^{(1-n)/2})$ is generic. The second part of the claim follows from lemma \ref{generic_ineq_lemma}.
	\end{proof}
	
	\subsection{The Potential Automorphy Theorem}
	Now we have collected all the necessary prerequisites to prove our potential automorphy theorem.
	
	\begin{theorem} \label{potential_automorphy_thm}
		Let $(\alpha_1, \dots, \alpha_n ; \beta_1, \dots, \beta_n)$ be a hypergeometric parameter of dimension $n$ modulo $d$ appearing in table \ref{special_hypergeos}.
		There exists an integer $C > 0$ with the following property. 
		Given the following objects:
		\begin{itemize}
			\item a CM field $F$ containing $\mu_d$,
			\item a prime $p \nmid d$,
			\item a prime $\ell \nmid p C$ such that $\ell \equiv 1 \pmod{d}$, 
			\item a continuous semisimple representation $\overline{\rho} : G_{F} \to \GL_n(\mathbb{F}_\ell)$ such that $\det \overline{\rho} = \overline{\chi}_{cyc}^{-n(n-1)/2}$,
			\item a finite extension $F^{avoid}/F$,
		\end{itemize}
		there exists a finite totally real extension $F_2/\rationals$, linearly disjoint from $F^{avoid}$ over $\rationals$ such that for $F' = F F_2$, there exists an isomorphism $\iota : \overline{\rationals}_\ell \cong \complex$ and a cuspidal automorphic representation $\Pi$ of $\GL_n(\mathbb{A}_{F'})$ which is $\iota$-ordinary of weight $0$ at all places above $\ell$ 
		such that 
		\[
			\overline{r_{\iota}(\Pi)} = \overline{\rho} \rvert_{G_{F'}}
		\]
		and for all places $v \mid p$ of $F'$
		\[
			WD(r_{\iota}(\Pi) \rvert_{G_{F_v'}})^{F-ss}
		\]
		is generic with monodromy operator having Jordan blocks $J_\alpha$ of size $\#\{i : \alpha_i = \alpha\}$ for $\alpha \in \ints/d\ints$. If $d$ is odd, then there also exists another cuspidal automorphic representation $\Pi_2$ of $\GL_n(\mathbb{A}_{F'})$ with the same properties as $\Pi$ except that the $v$-adic monodromy operator is nilpotent of rank $1$.
	\end{theorem}

	\begin{proof}
		Pick a prime $\lambda < \mathcal{O}_d$ above $\ell$ and choose an isomorphism $\iota : \overline{\rationals}_\ell \cong \complex$.
		Let $M = M(\alpha_1, \dots, \alpha_n ; \beta_1, \dots, \beta_n)$.
		Property (D) of proposition \ref{computed_motives_prop} shows that there is an isomorphism
		\[
			\alpha_\lambda : \det( M_{\lambda, 0} \otimes \psi_{\lambda} ) \cong \chi_{cyc}^{-n(n-1) / 2}
		\] 
		for some character $\psi_{\lambda} : G_{\rationals(\mu_d)} \to \mathcal{O}_{d,\lambda}^\times$.
		Let $E$ be a non-CM elliptic curve over $\rationals$.
		Denote by $r_{E,\ell'} : G_{\rationals} \to \GL_2(\rationals_{\ell'})$ the Galois representation acting on the dual of the $\ell'$-adic Tate module of $E$.
		As in the first part of the proof of \cite[Proposition 4.1]{qian_potential}, there exists $\ell' \equiv 1 \pmod{d}$ a large enough prime such that
		\begin{itemize}
			\item $\ell'$ is unramified in $F$,
			\item $\overline{r}_{E,\ell'}(G_{\tilde{F}}) = \GL_2(\mathbb{F}_{\ell'})$, where $\tilde{F}$ denotes the normal closure of $F/\rationals$,
			\item $\exists \sigma \in G_{F} \setminus G_{F(\mu_{\ell'})}$ such that $\overline{r}_{E,\ell'}(\sigma)$ is a scalar,
			\item $E$ has good ordinary reduction at $\ell'$.
		\end{itemize}
		Let $F^{av}$ be the normal closure of $F^{avoid} \overline{\rationals}^{\ker \overline{r}_{E,\ell'}}$ over $\rationals$.
		By \cite[Corollary 7.2.4]{10author}, there exists a finite totally real Galois extension
		$F^{suff} / \rationals$ and a finite Galois extension $F^{av}_2 / \rationals$
		such that 
		\begin{itemize}
			\item $F_{2}^{av}/\rationals$ is
			linearly disjoint from $F^{av}/\rationals$.
			\item $F^{suff}/\rationals$ is linearly disjoint from $F^{av}F_2^{av}/\rationals$.
			\item For any totally real extension $F'/F^{suff}$ which is linearly disjoint from $F_2^{av}$ over $\rationals$,
			the Galois representation $\Sym^{n-1} r_{E,\ell'} \rvert_{G_{F'}}$ is automorphic.
		\end{itemize}		
		Let $\lambda' < \mathcal{O}_d$ be a prime above $\ell'$ and $K/F$ a number field. Let $B$ be the etale $\mathcal{O}_d/\lambda \lambda'$-local system on $\Spec A_d \otimes_{\mathcal{O}_d} K$ corresponding to $M_\lambda/\lambda(\psi_\lambda) \times M_{\lambda'}/\lambda'(\psi_{\lambda'})$. Let $\widetilde{A}_d$ be defined as in lemma \ref{constant_det_lemma}. Consider the following moduli problem.
		\[
			S/ \Spec \widetilde{A}_d \otimes_{\mathcal{O}_d} K \mapsto \left\{ 
				\eta : B \rvert_{S} \cong \overline{\rho} \times \Sym^{n-1} \overline{r}_{E, \ell'} \mid \det(\eta) = \alpha_\lambda  \times \alpha_{\lambda'}
			\right\}
		\]
		It is represented by a finite etale scheme 
		\[
			X \to \Spec \widetilde{A}_d  \otimes_{\mathcal{O}_d} K.
		\] 
		If $K \to \complex$ is a field homomorphism, then 
		$X(\complex)$ is the covering space of $\Spec \widetilde{A}_d(\complex)$ corresponding to the $\pi_1(\complex \setminus \mu_d, t)$-set 
		\[
			\SL(M(s \alpha_i ; s \beta_j)_{\lambda,t}/\lambda) \times \SL(M(s \alpha_i ; s \beta_j)_{\lambda',t}/ \lambda')
		\]
		for some $s \in (\ints/d\ints)^\times$. 
		It follows from property (BM) of proposition \ref{computed_motives_prop} that
		$\pi_1(\complex \setminus \mu_d, t)$ acts transitively on this set, hence
		$X(\complex)$ is connected. Since $K \to \complex$ was arbitrary, $X$ is geometrically connected.
	
		Let $K = F F^{suff}$ and let $q$ be the composition 
		\[
		X \to \Spec \widetilde{A}_d \otimes_{\mathcal{O}_d} K \to \Spec A_d \otimes_{\mathcal{O}_d} K.
		\]
		Let $T = \Res_{K/\rationals} X$, $S_1 = \{\infty\}$, $S_2 = \emptyset$ and $S_3 = \{\ell, \ell', p\}$. We apply the theorem of Moret-Bailly \cite[Proposition 2.1]{HSBT} to $T/ \rationals$, $S_1$, $S_2$, $S_3$, $L = F^{suff} F^{av} F_2^{av}$ and
		\begin{itemize}
			\item if $u \in S_1$, then $\Omega_u = X(\reals \otimes_{\rationals} K) = X(\prod_{v \mid \infty} \complex)$,
			\item if $u \in \{\ell, \ell'\}$, then 
			\[
			\Omega_u = q^{-1} \{(x_{\tau})_{\tau : K \hookrightarrow \overline{\rationals}_u} \in \Spec A_d(\prod_{K \hookrightarrow \overline{\rationals}_u} \overline{\rationals}_u) : \forall\, \tau,\, |x_{\tau}|_{u} < 1 \},
			\]
			\item if $u = p$, then 
			\[
			\Omega_u = q^{-1} \{(x_{\tau})_{\tau : K \hookrightarrow \overline{\rationals}_u} \in \Spec A_d(\prod_{K \hookrightarrow \overline{\rationals}_u} \overline{\rationals}_u) : \forall\, \tau,\, |x_{\tau}|_{u} > 1 \}.
			\]            
		\end{itemize}
		
		Thus, there exists a finite totally real Galois extension $F_1/\rationals$ such that for $F' := K F_1 = F F_1 F^{suff}$, there exists $t \in \Spec \widetilde{A}_d(F')$ such that
		\begin{itemize}
			\item $F_1$ is linearly disjoint from $F^{suff} F^{av} F_2^{av}$ over $\rationals$,
			\item $M_{\lambda, t}/\lambda \otimes \psi_{\lambda} \cong \overline{\rho} \rvert_{G_{F'}}$ and  
			$M_{\lambda', t}/\lambda' \otimes \psi_{\lambda'} \cong \Sym^{n-1} \overline{r}_{E,\ell'} \rvert_{G_{F'}}$,
			\item $|t|_v > 1$ for all $v \mid p$,
			\item $M_{\lambda, t} \otimes \psi_{\lambda}$ and $M_{\lambda', t} \otimes \psi_{\lambda'}$ are ordinary of weight 0 by lemma \ref{ordinary_lemma}.
		\end{itemize} 
		Note that this also implies the following inclusions
		\[
		F' \cap F_2^{av} \subset K \cap F_2^{av} \subset F^{suff} F^{av} \cap F_2^{av} = \rationals
		\]
		and
		\[
			F_1 F^{suff} \cap F^{av} = \rationals.
		\] 
		Hence, $F'$ is linearly disjoint from $F_2^{av}$ over $\rationals$ and $F_2 := F_1 F^{suff}$ is linearly disjoint from $F^{av}$ over $F$. In particular $F_2$ is linearly disjoint from $F^{avoid} \subset F^{av}$ as claimed in the theorem.
		
		Since $(F')^+$ contains $F^{suff}$, and is linearly disjoint from $F_2^{av}$, we know that $\Sym^{n-1} r_{E,\ell'} \rvert_{G_{(F')^+}}$ is automorphic. By quadratic base change, the same holds for $\Sym^{n-1} r_{E,\ell'} \rvert_{G_{F'}}$.
		Since $\overline{\rationals}^{\ker \overline{r}_{E,\ell'}} \subset F^{av}$, we 
		have 
		\begin{itemize}
			\item $\overline{r}_{E,\ell'}(G_{F'}) = \GL_2(\mathbb{F}_{\ell'})$,
			\item $\exists \sigma \in G_{F'} \setminus G_{F'(\mu_{\ell'})}$ such that $\overline{r}_{E,\ell'}(\sigma)$ is a scalar.
		\end{itemize}
		By \cite[Lemma 7.1.6 (2)]{10author}, we know that $\Sym^{n - 1} \overline{r}_{E,\ell'}(G_{F'(\mu_{\ell'})})$ is enormous and
		 \cite[Lemma 2.6 (2)]{qian_potential} implies that $\Sym^{n - 1} \overline{r}_{E,\ell'} \rvert_{G_{F'}}$ is decomposed generic.
		
		The automorphy lifting theorem \cite[Theorem 6.1.2]{10author} implies that $M_{\lambda', t} \otimes \psi_{\lambda'}$ is automorphic.
		This yields a cuspidal cohomological automorphic representation $\Pi$ of $\GL_n(\mathbb{A}_{F'})$ such that $r_{\iota'}(\Pi) = M_{\lambda', t} \otimes \psi_{\lambda'}$ for some isomorphism $\iota' : \overline{\rationals}_{\ell'} \cong \complex$ such that $\iota' \circ \iota^{-1}$ fixes $\rationals(\mu_d)$. Since $M_{\lambda',t} \otimes \psi_{\lambda'}$ and $M_{\lambda,t} \otimes \psi_{\lambda}$ have the same Frobenius eigenvalues, this also implies that
		$r_{\iota}(\Pi) = M_{\lambda, t} \otimes \psi_{\lambda}$.
		Moreover, since $|t|_v > 1$ we can prove as in \cite[Lemma 1.15]{HSBT}, that the $v$-adic monodromy operator acting on $M_{\lambda,t}$ coincides with the monodromy around $\infty$ acting on $M_{B,t}$. This has the correct Jordan blocks by property (UM) of proposition \ref{computed_motives_prop}.
		By lemma \ref{purity_lemma}, we know that $WD(M_{\lambda,t} \rvert_{G_{F_w}})^{F-ss}$ is pure and hence generic by lemma \ref{pure_implies_generic_lemma}.
		
		If $d$ is odd, then we simply repeat the same proof but for $u = p$ replace $\Omega_u$ by
		\[
			\Omega_u = q^{-1} \{(x_{\tau})_{\tau : K \hookrightarrow \overline{\rationals}_u} \in \Spec A_d(\prod_{K \hookrightarrow \overline{\rationals}_u} \overline{\rationals}_u) : \forall\, \tau,\, |x_{\tau} - 1|_{u} < 1 \},
		\]
		to obtain $\Pi_2$.
	\end{proof}

\begin{corollary} \label{potential_automorphy_cor}
	Let $n \in \{3,4,5\}$ and $d_3 = 9$, $d_4 = 36 = \operatorname{lcm}(9,12,18)$, $d_5 = 168 = \operatorname{lcm}(21,24)$.
	There exists an integer $C > 0$ with the following property.
	Given the following objects:
	\begin{itemize}
		\item a nilpotent $n \times n$ matrix $N$,
		\item a CM field $F$ containing $\mu_{d_n}$,
		\item a prime $p \nmid d_n$,
		\item a prime $\ell \neq p$ such that $\ell \nmid C$ and $\ell \equiv 1 \pmod{d_n}$,
		\item a continuous semisimple representation $\overline{\rho} : G_{F} \to \GL_n(\mathbb{F}_{\ell})$ such that $\det \overline{\rho} = \overline{\chi}_{cyc}^{-n(n-1)/2}$,
		\item a finite Galois extension $F^{avoid}/F$,
	\end{itemize}
	there exists a finite totally real extension $F_2/\rationals$, linearly disjoint from $F^{avoid}$ over $\rationals$. For $F' = F F_2$, there exists a cuspidal automorphic representation $\Pi$ of $\GL_n(\mathbb{A}_{F'})$ which is ordinary of weight $0$ at all places above $\ell$ 
	such that 
	\[
	\overline{r_{\iota}(\Pi)} = \overline{\rho} \rvert_{G_{F'}}
	\]
	and for all places $v \mid p$ of $F'$
	\[
	WD(r_{\iota}(\Pi) \rvert_{G_{F_v'}})^{F-ss}
	\]
	is generic with monodromy operator conjugate to $N$.
\end{corollary}

\begin{proof}
	If $N = 0$, then this follows from the main theorem of \cite{qian_potential}.
	If $N \neq 0$, use proposition \ref{unipotent_monodromy_prop} to select the hypergeometric parameter corresponding to $N$ from table \ref{cor_table}, which is a subset of table \ref{special_hypergeos}. Then apply the theorem with this parameter.
\end{proof}

	 \begin{table}[ht]
	 \begin{tabular}{ c | c | c | c | c }
	 	$n$ & $d$ & $\alpha_i$ & $\beta_j$ & $\underline{c}$ \\ \hline
	 	$3$ & $9$ & $0, 0, 0$ & $1, 2, 6$ & $3, 7, 8$ \\
	 	
	 	$4$ & $9$ & $0, 0, 0, 0$ & $1, 4, 6, 7$ & $0, 0, 0$  \\
	 	$4$ & $12$ & $0, 0, 2, 2$ & $3, 4, 5, 10$ & $2, 11, 11$ \\  
	 	$4$ & $18$ & $0, 0, 0, 3$ & $4, 11, 16, 17$ & $1, 7, 10$ \\
	 	
	 	$5$ & $21$ & $0, 0, 0, 0, 0$ & $1, 2, 4, 15, 20$ & $6, 17, 19$ \\
	 	$5$ & $21$ & $0, 0, 0, 0, 1$ & $4, 10, 12, 18, 20$ & $1, 1, 19$ \\
		$5$ & $24$ & $0, 0, 0, 1, 1$ & $2, 6, 12, 19, 23$ & $0, 0, 0$ \\
		$5$ & $24$ & $0, 0, 1, 1, 7$ & $8, 12, 13, 17, 19$ & $0, 0, 0$ \\
		$5$ & $24$ & $0, 0, 0, 2, 6$ & $7, 8, 12, 19, 22$ & $0,0,0$ 
	 \end{tabular}
 \caption{Hypergeometric parameters sufficient for Corollary \ref{potential_automorphy_cor}}
 \label{cor_table}
 	\end{table}

\begin{corollary} \label{automorphic_existence_cor}
	Let $n \in \{3,4,5\}$ and $d_3 = 9$, $d_4 = 36$, $d_5 = 168$ and $p \nmid d_n$ a prime. For any nilpotent $n \times n$ matrix $N$, there exists a CM field $F'$ and a cuspidal cohomological automorphic representation $\Pi$ of $\GL_n(\mathbb{A}_{F'})$ of weight $0$, which is not essentially conjugate self-dual, such that the monodromy operator of $\operatorname{rec}_{F_v'}(\Pi_v |\det|^{(1-n)/2})$ coincides with $N$ for all $v \mid p$.
\end{corollary}

\begin{proof} 
	Let $\ell \equiv 1 \pmod{504}$ be a sufficiently large prime and $G = \PSL_2(\mathbb{F}_7)$.
	There is an irreducible non-self-dual representation $r : G \to \GL_3(\mathbb{F}_\ell)$. Let $L$ be the splitting field of $x^7 - 154 x + 99$.
	We have $\Gal(L/\rationals) \cong G$ by \cite{erbach79}.
	Consider $F = \rationals(\mu_{504 \ell})$ and  $\overline{\rho} : G_{F} \to \GL_3(\mathbb{F}_\ell)$ the restriction of $r$ and $F^{avoid} = \overline{\rationals}^{\ker \overline{\rho}}$. Now apply the previous corollary to $\overline{\rho}, \overline{\rho} \oplus \mathbf{1}$ and $ \overline{\rho} \oplus \mathbf{1} \oplus \mathbf{1}$. The resulting representation $\Pi$ cannot be essentially conjugate self-dual since that would imply the existence of a character 
	$\overline{\chi} : G_{F'} \to \overline{\mathbb{F}}_\ell^\times$ such that
	$\overline{\rho}^{\vee} \cong \overline{\rho} \otimes \overline{\chi}$.
	In particular, $\overline{\chi}$ must factor through the simple group $G$ and is trivial. Absurd.
	Finally, the claim follows from theorem \ref{varma_thm}.
\end{proof}

	\subsection{Local-Global Compatibility}
		In this section we explain the consequences of the potential automorphy theorem for the compatibility of local and global Langlands correspondences.
		Fix an isomorphism $\iota : \overline{\rationals}_\ell \cong \complex$. First we need a variant of \cite[Thm 6.1.2]{10author}. 
		
	\begin{lemma} \label{lifting_level_lemma}
		Let $F$ be an imaginary CM or totally real field, let $c \in \Aut(F)$ be complex conjugation, and let $p$ be a prime. Let $w \nmid p$ be a place of $F$ and $K_w \subset \GL_n(F_{w})$ an open compact subgroup containing the Iwahori subgroup. If $\rho : G_{F} \to \GL_n(\overline{\rationals}_\ell)$ is a continuous representation satisfying the following conditions:
		\begin{enumerate}[(1)]
			\item $\rho$ is unramified almost everywhere.
			\item For each place $v \mid \ell$ of $F$, the representation $\rho\rvert_{G_{F_v}}$ is potentially semi\-stable, ordinary with regular Hodge--Tate weights.
			\item $\overline{\rho}$ is absolutely irreducible and decomposed generic \cite[4.3.1]{10author}. The image of $\overline{\rho} \rvert_{F(\zeta_\ell)}$ is enormous \cite[6.2.28]{10author}.
			\item There exist $\sigma \in G_{F} \setminus G_{F(\zeta_\ell)}$ such that $\overline{\rho}(\sigma)$ is a scalar. We have $\ell > n$.
			\item There exists a regular algebraic cuspidal automorphic representation $\pi$ of $\GL_n(\mathbb{A}_F)$ and an isomorphism $\iota : \overline{\rationals}_\ell \to \complex$ such that $\pi$ is $\iota$-ordinary and $\overline{r_{\iota}(\pi)} \cong \overline{\rho}$.
			\item The restriction $\overline{\rho} \rvert_{G_{F_w}}$ is trivial and $q_w \equiv 1 \pmod{\ell}$.
			\item The space of invariants $\pi_w^{K_w}$ contains a non-zero vector and the monodromy operators of $r_{\iota}(\pi) \rvert_{G_{F_w}}$ and $\rho \rvert_{G_{F_w}}$ are conjugate to each other. 
		\end{enumerate}
		Then $\rho$ is ordinarily automorphic of weight $\iota \lambda$: there exists an $\iota$-ordinary cuspidal automorphic representation $\Pi$ of $\GL_n(\mathbb{A}_F)$ of weight $\iota \lambda$ such that $\rho \cong r_{\iota}(\Pi)$ and $\Pi_w^{K_w} \neq 0$. Moreover, if $v \nmid p$ is a finite place of $F$ and both $\rho$ and $\pi$ are unramified at $v$, then $\Pi_v$ is unramified. 
	\end{lemma}

	\begin{proof}
		Given \cite[Thm 6.1.2]{10author}, the only new thing we need to prove is that $\Pi_w^{K_w} \neq 0$. As in \cite[\S 6.6]{10author}, after a solvable base change which is totally split at $w$, it suffices to prove this under the additional conditions listed in \cite[\S 6.6.1]{10author}. 
		
		Consider the patched homology complex $\mathcal{C}_\infty$ of $S_\infty$-modules used in the proof of theorem \cite[Thm 6.6.2]{10author} and the $S_\infty$-algebras $R_\infty$ and $T_\infty \subset \End_{\mathbf{D}(S_\infty)}(\mathcal{C}_\infty)$. By replacing the Iwahori subgroup in the level subgroups defining $\mathcal{C}_\infty$ with the group $K_w$ we obtain a similar bounded complex of $S_\infty$-modules  $\mathcal{C}_\infty(K_w)$ and an $S_\infty$-algebra
		$T_\infty(K_w) \subset \End_{\mathbf{D}(S_\infty)}(\mathcal{C}_\infty(K_w))$.
		Moreover, there is a natural morphism of complexes $\mathcal{C}_\infty \to \mathcal{C}_\infty(K_w)$ and a compatible map of $S_\infty$-algebras $T_\infty \to T_\infty(K_w)$.
		
		The Galois representation $\rho$ defines a point $y$ of $\Spec R_\infty$. Denote its image in $\Spec S_\infty$ by $\mathfrak{p}$. We wish to show that 
		$y$ lies in the support of $H^\bullet( \mathcal{C}_\infty(K_w) \otimes_{S_\infty}^{\mathbb{L}} S_\infty/\mathfrak{p})[1/p]$, since then the lemma follows as in the proof of \cite[Thm 6.6.2]{10author}.
		Arguing as in the proof of \cite[Corollary 6.3.9]{10author}, we see that it suffices to show that $y$ lies in the support of $H^\bullet(\mathcal{C}_\infty(K_w))$ or equivalently that $y \in \Spec T_\infty(K_w) \subset \Spec R_\infty$.
		
 		Let $x \in \Spec R_\infty$ be the point corresponding to the Galois representation $r_\iota(\pi)$. 
		As in the beginning of the proof of \cite[prop 6.3.8]{10author}, the key lemma \cite[Lemma 6.2]{calegari_geraghty_beyond} implies that
		$\Supp_{S_{\infty,x}} H^\bullet(\mathcal{C}_{\infty}(K_w)_x)$ contains a point of codimension at most $l_0$. Let $x_1$ denote a preimage of such a point in $\Spec T_\infty(K_w)$. If we view $\Spec T_\infty(K_w)$ as a closed subset of $\Spec R_\infty$, then $x_1$ must be a generic point of $\Spec R_\infty$ of maximal dimension.
		
		Let $R$ be the set of places used in \cite[\S 6.6.1]{10author}. In particular, $w \in R$. The proofs of \cite[Prop 3.1]{taylor_ladic_part2} and \cite[Lemma 6.2.26]{10author} show that the irreducible components of maximal dimension of $\Spec R_\infty$ are in bijection with tuples $(N_v)_{v \in R}$ where $N_v$ is a conjugacy class of a $n \times n$ nilpotent matrix. Moreover, a point of $\Spec R_\infty$ corresponding to a Galois representation $\rho' : G_{F} \to \GL_n(\overline{\rationals}_\ell)$ lies on the irreducible component of maximal dimension indexed by $N_v = N(\rho' \rvert_{G_{F_v}})$.	
		
		If $R = \{w\}$, then we use assumption (7) to directly conclude that $y$ lies in the closure of $x_1$, hence $y \in \Spec T_\infty(K_w)$. If $\# R > 1$, then we consider the deformation datum $\mathcal{S}_\chi^{ord}$, where $\chi_{w,i} = 1$ and for $v \in R \setminus \{w\}$ we choose $\chi_{v,1}, \dots, \chi_{v, 1} : \mathcal{O}_{F_v}^\times \to \mathcal{O}^\times$ pairwise distinct characters which are trivial mod $\varpi$. Similarly to the proof of \cite[Thm 6.6.2]{10author}, this leads to a setup $R_\infty^{'}, \mathcal{C}_\infty^{'}(K_w), T_\infty^{'}(K_w)$ with an isomorphism $R_\infty'/\varpi \cong R_\infty/\varpi$ which is compatible with an isomorphism $\mathcal{C}_\infty'(K_w) \otimes^{\mathbb{L}}_{S_\infty} S_\infty/\varpi \cong \mathcal{C}_\infty(K_w) \otimes_{S_\infty}^{\mathbb L} S_\infty/\varpi$.
		Moreover, \cite[Lemma 6.3.7]{10author} applies to $R_\infty^{'}, \mathcal{C}_\infty^{'}(K_w), T_\infty^{'}(K_w)$. The irreducible components of $\Spec R'_\infty$ of maximal dimension are indexed by a single nilpotent conjugacy class corresponding to the monodromy at $w$.
		
		Let $y_1$ be the maximal dimension generic point of $\Spec R_\infty$ containing $y$. Choose generic points $\overline{x}_1$ and $\overline{y}_1$ of $\Spec R_\infty/(x_1, \varpi)$ and $\Spec R_\infty(y_1, \varpi)$. By assumption on the monodromy at $w$, we see that both these points lie on the irreducible component of $R_\infty'$ indexed by $N_w = N(\rho \rvert_{G_{F_w}})$, i.e. both $\overline{x}_1$ and $\overline{y}_1$ generalise to the same generic point $x' \in \Spec R_\infty'$. Finally, we deduce in exactly the same way as at the end of the proof of \cite[Prop 6.3.8]{10author} that $y_1 \in \Spec T_\infty(K_w)$, hence also $y \in \Spec T_\infty(K_w)$, as desired.
	\end{proof}
	
	\begin{lemma} \label{monodromy_bound_level_lemma}
		Let $K/\rationals_p$ be a finite extension with residue field $k$ and $\pi$ a generic (i.e. admits a Whittaker model) smooth irreducible representation of $\GL_n(K)$ such that
		$\pi^{\operatorname{Iw}} \neq 0$, where $\operatorname{Iw} = \operatorname{Iw}_n \subset \GL_n(K)$ is the Iwahori subgroup. Let $m_1 + \dots + m_s = n$ be a partition. Let $P < \GL_n$ be the standard parabolic subgroup corresponding to the conjugate partition $n_j = \# \{i : m_i \geq j\}$ and
		$U = \{g \in \GL_n(\mathcal{O}_K) : \overline{g} \in P(k)\} < \GL_n(K)$.
		\begin{enumerate}[(a)]
			\item If the monodromy operator of $\operatorname{rec}_K(\pi)$ has Jordan block sizes given by $m_1, \dots, m_s$, then $\pi^U \neq 0$.
			\item If $\pi^{U} \neq 0$, then the rank of the monodromy operator of $\operatorname{rec}_K(\pi)$ is at most $n - s$.
		\end{enumerate}
	\end{lemma}

	\begin{proof}
		As $\pi^{\operatorname{Iw}} \neq 0$, it follows from \cite[Proposition 2.6]{casselman80}
		that the supercuspidal support of $\pi$ consists of unramified characters of $K^\times$.
		
		Since $\pi$ is generic, \cite[Thm 9.7]{zelevinsky80} implies that there are segments $\Delta_1, \dots, \Delta_t$
		of lengths $r_i$ such that $\pi = \langle \Delta_1 \rangle \times \dots \times \langle \Delta_t \rangle$ and $\Delta_i, \Delta_j$ are not linked for $i \neq j$. See \cite{zelevinsky80} for the definitions of these terms.
		As explained in \cite[\S 4.4]{rodier82}, it follows from \cite[Thm VII.2.20]{harris_taylor} that
		\[
		\operatorname{rec}_{K}(\pi) = \bigoplus_{i = 1}^t \chi_i \otimes \Sp_{r_i} 
		\]
		for some unramified characters $\chi_i$. 
		
		Let $Q < \GL_n$ be the standard parabolic subgroup corresponding to the partition 
		$r_1 + \dots + r_t = n$. Consider the Bruhat 
		decomposition 
		\[
		\GL_n(K) = \bigcup_{w \in W_Q \backslash W / W_P } Q(K) w U,
		\]
		where $W$ is the Weyl group of $\GL_n$ and $W_Q$ and $W_P$ are defined as the Weyl groups of the Levi subgroups of $Q$ and $P$.
		By definition of normalised induction, we have
		\[
		(\langle \Delta_1 \rangle \times \dots \times \langle \Delta_t \rangle)^{\operatorname{U}} \cong \bigoplus_{w \in  W_Q \backslash W / W_P  } (\langle \Delta_1 \rangle \otimes \dots \otimes \langle \Delta_t \rangle)^{Q(K) \cap w U w^{-1}}.
		\]
		Now we can prove part (a). The assumption on the monodromy of $\operatorname{rec}_K(\pi)$ implies that the partition $r_1 + \dots + r_t = n$ is equivalent to $m_1 + \dots + m_s = n$.
		
		Hence, the choice of $P$ implies that there exists an element $w \in W$ such that $M \cap w P w^{-1} = M \cap B$, where $M$ is the Levi subgroup of $Q$ and $B$ is the group of upper triangular matrices. Namely, if we assume $m_1 \geq m_2 \geq \dots \geq m_s$, then we can take 
		$w(x) = j + \sum_{i < y} m_{i}$, where
		$x = y + \sum_{i < j} n_i$ and $y \leq n_j$.
		With this choice of $w$ we have $M(K) \cap w U w^{-1} = \operatorname{Iw}_{m_1} \times \dots \times \operatorname{Iw}_{m_s}$ and in particular, $\pi^U \neq 0$.
		
		For part (b) it suffices to show that $t \geq s$. The first proof of \cite[Proposition 2.3]{casselman80} implies that $\langle \Delta_i \rangle^{\operatorname{Iw}_{r_i}} \cong \phi_{w_0} \complex$, where $w_0$ is the longest Weyl group element of $\GL_{r_i}$ and $\phi_{w_0}$ is the function defined in \cite[\S 2]{casselman80}. Any simple reflection acts non-trivially on $\phi_{w_0}$. In particular, if $U < \GL_{r_i}(K)$ is a parahoric subgroup such that $\langle \Delta_i \rangle^U \neq 0$, then $U$ is conjugate to $\operatorname{Iw}_{r_i}$.

		Hence, $(\langle \Delta_1 \rangle \times \dots \times \langle \Delta_t \rangle)^{\operatorname{U}} \neq 0$ implies that 
		$M(K) \cap w U w^{-1}$ equals $\operatorname{Iw}_{r_1} \times \dots \times \operatorname{Iw}_{r_t}$ for some choice of $w$. 
		If $t < s$, then the first block of $P$, which is of size $s$, is larger than the number of factors of $M$. Thus, for any $w$, the group $P \cap w M w^{-1}$ has a quotient isomorphic to $\GL_{j}$ for some $j > 1$. In particular we cannot have $M \cap B = M \cap w P w^{-1}$ and $M(K) \cap w U w^{-1} = \operatorname{Iw}_{r_1} \times \dots \times \operatorname{Iw}_{r_s}$.
	\end{proof}
	
	\begin{theorem} \label{local_global_thm}
		Let $n \in \{3,4,5\}$ and $d_3 = 9$, $d_4 = 36$, $d_5 = 168$.
		There exists an integer $C > 0$ with the following property.
		Given the following objects:
		\begin{itemize}
			\item a CM field $F$ containing $\mu_{d_n}$,
			\item a prime $\ell$ such that $\ell \nmid C$ and $\ell \equiv 1 \pmod{d_n}$  and an isomorphism $\iota : \overline{\rationals}_\ell \cong \complex$,
			\item a cuspidal cohomological automorphic representation $\Pi$ of $\GL_n(\mathbb{A}_F)$ which is $\iota$-ordinary at all places above $\ell$,
		\end{itemize}
		such that
		\begin{itemize}
			\item $\overline{r_{\iota}(\Pi)}(G_{F(\mu_\ell)})$ is enormous \cite[6.2.28]{10author} and $\overline{r_{\iota}(\Pi)}$ is absolutely irreducible and decomposed generic \cite[4.3.1]{10author},
			\item there exists $\sigma \in G_{F} \setminus G_{F(\mu_\ell)}$ such that $\overline{r_{\iota}(\Pi)}(\sigma)$ is a scalar,
			\item there exists $\gamma \in \GL_n(\overline{\mathbb{F}}_\ell)$ such that $\gamma \overline{r_\iota(\Pi)}(G_{F}) \gamma^{-1} \subset \GL_n(\mathbb{F}_\ell)$,
			\item $\det \overline{r_{\iota}(\Pi)} = \overline{\chi}_{cyc}^{-n(n-1)/2}$,
		\end{itemize}
		we have
		\[
			WD(r_{\iota}(\Pi) \rvert_{G_{F_v}})^{F-ss} \cong \operatorname{rec}_{F_v}(\Pi_v |\det |^{(1-n)/2}),
		\]
		for all places $v \nmid d_n \ell$.
	\end{theorem}

	\begin{proof}
		Let $v \nmid \ell$ be a place of $F$ dividing $p \nmid d_n$. With the theorem of Varma (theorem \ref{varma_thm}), it suffices to prove that $WD(r_{\iota}(\Pi) \rvert_{G_{F_v}})^{F-ss}$ is generic. Let $N$ denote the monodromy operator of $WD(r_{\iota}(\Pi) \rvert_{G_{F_v}})^{F-ss}$.
		
		Let $F^{avoid}$ be the normal closure of $\overline{F}^{\ker \overline{r_\iota(\Pi)}}(\mu_\ell)$ over $\rationals$.
		We apply Corollary \ref{potential_automorphy_cor} to find a finite totally real extension $F_2/\rationals$ linearly disjoint from $F^{avoid}$ over $\rationals$ such that for $F' = F F_2$ we have a cuspidal automorphic representation $\Pi'$ of $\GL_n(\mathbb{A}_{F'})$ such that 
		\[
			\overline{r_{\iota}(\Pi)} \rvert_{G_{F'}} = \overline{r_{\iota}(\Pi')}
		\] 
		and the Weil--Deligne representation
		\[
			WD(r_{\iota}(\Pi') \rvert_{G_{F'_{v'}}} )^{F-ss} 
		\]
		is generic with monodromy operator conjugate to $N$ for all places $v' \mid p$ of $F'$. 
		
		By \cite[Lemma 7.1.7]{10author}, $\overline{r_{\iota}(\Pi)} \rvert_{G_{F'}}$ is decomposed generic and absolutely irreducible. 
		Moreover, since $F'(\mu_\ell)$  is linearly disjoint from $F^{avoid}$ over $F(\mu_\ell)$, we find that $\overline{r_{\iota}(\Pi)}(G_{F'(\mu_\ell)}) = \overline{r_{\iota}(\Pi)}(G_{F(\mu_\ell)})$ is enormous.
		After a cyclic base change we may moreover assume that $\Pi'_{v'}$ has Iwahori fixed vectors by \cite[5.2.3 and 5.2.4]{acampo_rigidity}, that $q_{v'} \equiv 1 \pmod{\ell}$ and that $\overline{r_{\iota}(\Pi)} \rvert_{G_{F'_{v'}}}$ is trivial.
		Theorem \ref{varma_thm} implies that 
		\[
		\operatorname{rec}_{F'_{v'}}(\Pi'_{v'} |\det|^{(1-n)/2}) = WD(r_{\iota}(\Pi') \rvert_{G_{F'_{v'}}} )^{F-ss}.
		\]
		Hence, part (a) of lemma \ref{monodromy_bound_level_lemma} implies that  $(\Pi'_{v'})^{K_{v'}} \neq 0$, where $K_{v'} = \{k \in \GL_n(\mathcal{O}_{F'_{v'}}) : \overline{k} \in P(k(v'))\}$ and $P \subset \GL_n$ is the standard parabolic subgroup corresponding to the conjugate partition of the Jordan block sizes of $N$.
		
		By lemma \ref{lifting_level_lemma}, there exists another automorphic representation $\Pi''$ of $\GL_n(\mathbb{A}_{F'})$ such that 
		\[
			r_{\iota}(\Pi'') = r_{\iota}(\Pi) \rvert_{G_{F'}}
		\]
		and $(\Pi''_{v'})^{K_{v'}} \neq 0$. Part (b) of lemma \ref{monodromy_bound_level_lemma} together with the main theorem of \cite{varma_local_global} implies that the rank of the monodromy operator of the Weil--Deligne representation $WD(r_\iota(\Pi'')\rvert_{G_{F'_{v'}}})$ is as big as possible and that local-global compatibility holds for $\Pi''$. It follows from theorem \ref{varma_thm} that $WD(r_{\iota}(\Pi'')) = WD(r_{\iota}(\Pi) \rvert_{G_{F'}})$ is generic. Thus, $WD(r_{\iota}(\Pi))$ is generic, too. 
	\end{proof}
	
	\section{Searching Hypergeometric Parameters}

	\subsection{The Computational Check} In this section we describe a few explicit numerical criteria which follow from the results stated in \S 2. We have implemented these in a straightforward python program (available at \url{https://github.com/LAC1213/dworkmotives}) to produce the hypergeometric parameters listed in table \ref{special_hypergeos}.
	
	\begin{num}
		Fix the following data.
		\begin{itemize}
			\item Positive integers $d > n$.
			\item A hypergeometric parameter $(\alpha_1, \dots, \alpha_n ; \beta_1, \dots, \beta_n)$ modulo $d$.
			\item A vector $\underline{c} \in (\ints/d\ints)^3$ such that $c_0 + c_1 + c_2 = 0$.
			\item A subgroup $U < (\ints/d\ints)^\times$.
		\end{itemize}
		Let $M = M(\alpha_1, \dots, \alpha_n ; \beta_1, \dots, \beta_n)$.
	\end{num}

	\begin{num}
		It follows from Proposition \ref{hodge_numbers_prop} that $M$ satisfies (R) if and only if for every $s \in (\ints/d\ints)^\times$, the following $n$ integers are pairwise distinct
		\[
			\sum_{i} [s \beta_j - s \alpha_i] - \sum_{i} [ s \beta_j - s \beta_i] \qquad j = 1, \dots, n.
		\] 
		Equivalently, they form a set of cardinality $n$. If this condition is satisfied we set $R(\alpha_i ; \beta_j) = \True$, otherwise $R(\alpha_i; \beta_j) = \False$. If $n$ is small and $d$ grows, then computing $R(\alpha_i ; \beta_j)$ takes $O(d)$ operations. 
	\end{num}

	\begin{num}
		We define $BM(\alpha_i ; \beta_j) = \True$ if the conditions of proposition \ref{big_monodromy_prop} are satisfied. For small $n$ this requires $O(d)$ operations to compute. 
	\end{num}
	
	\begin{num}
		We define $BM_{fin}(\alpha_i ; \beta_j ; U) = \True$ if additionally the conditions of proposition \ref{big_monodromy_finite_prop} are satisfied. Note that this requires $O(d)$ operations to compute. 
	\end{num}

	\begin{num}
		We define $D(\alpha_i ; \beta_j ; \underline{c}) = \True$ if the conditions of proposition \ref{determinant_nth_power_prop} are satisfied. To check for the existence of the $x_\epsilon \in \ints$, we do linear algebra in the vector space spanned by the functions in $E(d)$. Namely we have to invert the matrix given by the functions of a set of basis elements in $E(d)$. This costs $O(d^3)$ operations.
	\end{num}

	\begin{num}
		Let $p_1 \geq p_2 \geq \dots \geq p_r \geq 1$ be a partition of $n$. Here is a simple algorithm which finds hypergeometric parameters $(\alpha_i ; \beta_j)$ modulo $d$ such that the properties (BM), (R), (D) stated in proposition \ref{computed_motives_prop} are satisfied and the monodromy operator at $t = \infty$ has Jordan blocks sizes $p_1, \dots, p_r$.
	
		\begin{enumerate}
			\item Enumerate all possible hypergeometric parameters $(\alpha_i ; \beta_j)$ modulo $d$
			of the form
			\[
				(\underbrace{0 , \dots, 0}_{p_1}, \underbrace{\gamma_2, \dots, \gamma_2}_{p_2}, \dots, \underbrace{\gamma_r, \dots, \gamma_r}_{p_r} ; \beta_1, \dots, \beta_n).
			\]
			There are $O(d^{n + r - 1})$ of these.
			\item For each such $(\alpha_i ; \beta_j)$ compute 
			\[
				x = R(\alpha_i ; \beta_j) \wedge BM(\alpha_i ; \beta_j).
			\]
			\item If $x = \True$, then enumerate all possibilities for $\underline{c} \in (\ints/d\ints)^3$. If $D(\alpha_i ; \beta_j ; \underline{c})$ is satisfied for any of them, then output $\alpha_i ; \beta_j ; \underline{c}$.
		\end{enumerate}
	\end{num}
	
	\subsection{Results of the Program} We have run this algorithm for $d \leq 30$ and partitions of $n \in \{4,5,6\}$ which are not of the form $(2,1, \dots, 1)$ or $(1,1, \dots, 1)$. (For the excluded partitions one can simply use a family of motives with $d$ odd and partition $(n)$.) We have listed some of the resulting parameters in table \ref{special_hypergeos}.
	
	\begin{table}[ht] 
		
		\begin{center}
			\begin{tabular}{ c | c | c | c | c | c }
				$n$ & $d$ & $\alpha_i$ & $\beta_j$ & $\underline{c}$ & $U < (\ints/d\ints)^\times$ \\ \hline
				$3$ & $9$ & $0, 0, 0$ & $1, 2, 6$ & $3, 7, 8$ & $ \{\pm 1\} $  \\
				
				$4$ & $9$ & $0, 0, 0, 0$ & $1, 4, 6, 7$ & $0, 0, 0$ & $(\ints/9\ints)^\times$ \\
				$5$ & $9$ & $0, 0, 0, 0, 0$ & $1, 2, 3, 4, 8$ & $5, 5, 8$ & $\{\pm 1\}$ \\
				$6$ & $12$ & $0, 0, 0, 0, 0, 0$ & $1, 2, 3, 6, 8, 10$ & $4, 9, 11$ & $\{\pm 1\}$ \\
				$4$ & $12$ & $0, 0, 2, 2$ & $3, 4, 5, 10$ & $2, 11, 11$ & $(\ints/12\ints)^\times$\\
				$5$ & $14$ & $0, 0, 0, 1, 1$ & $2, 4, 7, 11, 13$ & $0, 0, 0$ & $ \{\pm 1\} $ \\
				$4$ & $18$ & $0, 0, 0, 3$ & $4, 11, 16, 17$ & $1, 7, 10$ & $ \{\pm 1\} $\\
				$4$ & $20$ & $0, 0, 2, 2$ & $3, 4, 9, 18$ & $2, 19, 19$ & $(\ints/20\ints)^\times$ \\
				$6$ & $20$ & $0, 0, 0, 2, 2, 2$ & $3, 4, 6, 9, 16, 18$ & $4, 17, 19$ & $(\ints/20\ints)^\times$ \\
				$4$ & $20$ & $0, 0, 0, 2$ & $3, 4, 9, 16$ & $2, 19, 19$ & $ \{\pm 1\} $ \\
				$6$ & $20$ & $0, 0, 0, 0, 1, 1$ & $2, 4, 8, 10, 11, 17$ & $4, 18, 18$ & $ \{\pm 1\} $\\
				$6$ & $20$ & $0, 0, 0, 0, 0, 2$ & $3, 10, 12, 13, 16, 18$ & $1, 8, 11$ & $(\ints/20\ints)^\times$ \\
				$6$ & $20$ & $0, 0, 0, 0, 1, 3$ & $4, 10, 11, 13, 17, 19$ & $1, 3, 16$ & $ \{\pm 1\} $ \\
				$5$ & $21$ & $0, 0, 0, 0, 0$ & $1, 2, 4, 15, 20$ & $6, 17, 19$ & $ \{\pm 1\} $\\
				$5$ & $21$ & $0, 0, 0, 0, 1$ & $4, 10, 12, 18, 20$ & $1, 1, 19$ & $ \{\pm 1\} $\\
				$5$ & $22$ & $0, 0, 0, 1, 1$ & $2, 6, 11, 17, 21$ & $0, 0, 0$ & $ \{\pm 1\} $ \\
				$5$ & $24$ & $0, 0, 0, 1, 1$ & $2, 6, 12, 19, 23$ & $0, 0, 0$ & $ \{\pm 1\} $\\
				$5$ & $24$ & $0, 0, 1, 1, 7$ & $8, 12, 13, 17, 19$ & $0, 0, 0$ & $ \{\pm 1\} $\\
				$5$ & $24$ & $0, 0, 0, 2, 6$ & $7, 8, 12, 19, 22$ & $0,0,0$ & $ \{\pm 1, \pm 11\} $
			\end{tabular}
			
		\end{center}
		\caption{some special hypergeometric parameters}
		\label{special_hypergeos}
	\end{table}

	\begin{num}
    Given a partition of $n \in \{4,5,6\}$, our program computed those values of $d \leq 30$ for which there exists a hypergeometric parameter satisfying (BM), (R), and (D). 
    We display this data in table \ref{possible_d_table}. We omitted partitions of the form $(n), (2,1,\dots,1)$ or $(1,1,\dots,1)$ since for these one can use any odd value of $d$ as shown in \cite{qian_potential}. For the partitions $(2,2,2), (3,2,1), (2,2,1,1), (3,1,1,1)$ our program did not find any hypergeometric parameters satisfying (BM), (R), and (D).
	\end{num}
	
	\begin{table}[ht]
		\begin{tabular}{ c | l | l }
			$n$ & $\{p_i\}$ & $d$ \\ \hline
			$4$ & $2,2$ & $12,20,24$ \\
			$4$ & $3,1$ & $18,20,24,28,30$ \\
			$5$ & $3,2$ & $12,14,16,18,22,24,26,28$ \\
			$5$ & $4,1$ & $18,21,24$ \\
			$5$ & $2,2,1$ & $18,24$ \\
			$5$ & $3,1,1$ & $24$ \\
			$6$ & $3,3$ & $20,24,30$ \\
			$6$ & $4,2$ & $20,24,28$ \\
			$6$ & $5,1$ & $20,24,30$ \\
			$6$ & $4,1,1$ & $20,24$
		\end{tabular}
		\caption{values of $d \leq 30$ for which the program found a hypergeometric parameter satisfying (BM), (R), (D)}
		\label{possible_d_table}
	\end{table}
	
	\begin{num} Unfortunately, our algorithm above is very slow, hence we only ran it for $d \leq 30$. Moreover, for any fixed partition we have no theoretical understanding for which $d$ we can expect a suitable hypergeometric parameter to exist.
	Thus, further investigations are necessary. It is not clear if one has to use varieties beyond Dwork hypersurfaces to find all the families of motives satisfying (BM), (R), (UM).
	\end{num}

    \bibliographystyle{alpha}
	\bibliography{refs}

@misc{matsumoto2024potential,
	title={On the potential automorphy and the local-global compatibility for the monodromy operators at $p \neq l$ over CM fields}, 
	author={Kojiro Matsumoto},
	year={2024},
	eprint={2312.01551},
	archivePrefix={arXiv},
	primaryClass={math.NT}
}

@Article{BL_even,
	Author = {Barnet-Lamb, Thomas},
	Title = {Potential automorphy for certain {Galois} representations to {{\(\mathrm{GL}_{2n}\)}}},
	FJournal = {Journal f{\"u}r die Reine und Angewandte Mathematik},
	Journal = {J. Reine Angew. Math.},
	ISSN = {0075-4102},
	Volume = {674},
	Pages = {1--41},
	Year = {2013},
	Language = {English},
	DOI = {10.1515/crelle.2012.032},
	zbMATH = {6137863},
	Zbl = {1310.11059}
}

@Article{BL_odd,
	Author = {Barnet-Lamb, Thomas},
	Title = {On the potential automorphy of certain odd-dimensional {Galois} representations},
	FJournal = {Compositio Mathematica},
	Journal = {Compos. Math.},
	ISSN = {0010-437X},
	Volume = {146},
	Number = {3},
	Pages = {607--620},
	Year = {2010},
	Language = {English},
	DOI = {10.1112/S0010437X09004527},
	zbMATH = {5722159},
	Zbl = {1235.11050}
}

@Book{deligne_milne_ogus_shih,
	Author = {Deligne, Pierre and Milne, James S. and Ogus, Arthur and Shih, Kuang-yen},
	Title = {Hodge cycles, motives, and {Shimura} varieties},
	FSeries = {Lecture Notes in Mathematics},
	Series = {Lect. Notes Math.},
	ISSN = {0075-8434},
	Volume = {900},
	Year = {1982},
	Publisher = {Springer, Cham},
	Language = {English},
	zbMATH = {3728195},
	Zbl = {0465.00010}
}

@Misc{deligne_valeurs,
	Author = {Deligne, Pierre},
	Title = {Values of {{\(L\)}}-functions and periods of integrals},
	Year = {1979},
	Language = {French},
	HowPublished = {Automorphic forms, representations and {L}-functions, {Proc}. {Symp}. {Pure} {Math}. {Am}. {Math}. {Soc}., {Corvallis}/{Oregon} 1977, {Proc}. {Symp}. {Pure} {Math}. 33, {No}. 2, 313-346 (1979).},
	zbMATH = {3702777},
	Zbl = {0449.10022}
}

@Article{HSBT,
	Author = {Harris, Michael and Shepherd-Barron, Nick and Taylor, Richard},
	Title = {A family of {Calabi}-{Yau} varieties and potential automorphy},
	FJournal = {Annals of Mathematics. Second Series},
	Journal = {Ann. Math. (2)},
	ISSN = {0003-486X},
	Volume = {171},
	Number = {2},
	Pages = {779--813},
	Year = {2010},
	Language = {English},
	DOI = {10.4007/annals.2010.171.779},
	URL = {annals.princeton.edu/annals/2010/171-2/p04.xhtml},
	zbMATH = {5712744},
	Zbl = {1263.11061}
}

@Article{scholze_perfectoid,
	Author = {Scholze, Peter},
	Title = {Perfectoid spaces},
	FJournal = {Publications Math{\'e}matiques},
	Journal = {Publ. Math., Inst. Hautes {\'E}tud. Sci.},
	ISSN = {0073-8301},
	Volume = {116},
	Pages = {245--313},
	Year = {2012},
	Language = {English},
	DOI = {10.1007/s10240-012-0042-x},
	zbMATH = {6120994},
	Zbl = {1263.14022}
}

@book{katz_gkm,
	author = {Katz, Nicholas M.},
	title = {{G}auss sums, {K}loosterman sums, and monodromy groups},
	Year = {1988},
	Publisher = {Princeton University Press},
	Language = {English}
}

@InCollection{katz_another_look,
	Author = {Katz, Nicholas M.},
	Title = {Another look at the {Dwork} family},
	BookTitle = {Algebra, arithmetic, and geometry. In honor of Y. I. Manin on the occasion of his 70th birthday. Vol. II},
	ISBN = {978-0-8176-4746-9; 978-0-8176-4747-6},
	Pages = {89--126},
	Year = {2009},
	Publisher = {Boston, MA: Birkh{\"a}user},
	Language = {English},
	DOI = {10.1007/978-0-8176-4747-6_4},
	zbMATH = {5778253},
	Zbl = {1195.14015}
}

@Article{beukers_heckman,
	Author = {Beukers, F. and Heckman, G.},
	Title = {Monodromy for the hypergeometric function {{\(_ nF_{n-1}\)}}},
	FJournal = {Inventiones Mathematicae},
	Journal = {Invent. Math.},
	ISSN = {0020-9910},
	Volume = {95},
	Number = {2},
	Pages = {325--354},
	Year = {1989},
	Language = {English},
	DOI = {10.1007/BF01393900},
	zbMATH = {4084031},
	Zbl = {0663.30044}
}

@Article{griffiths_periods,
	Author = {Griffiths, Phillip A.},
	Title = {On the periods of certain rational integrals. {I}, {II}},
	FJournal = {Annals of Mathematics. Second Series},
	Journal = {Ann. Math. (2)},
	ISSN = {0003-486X},
	Volume = {90},
	Pages = {460--495, 496--541},
	Year = {1969},
	Language = {English},
	DOI = {10.2307/1970746},
	zbMATH = {3341218},
	Zbl = {0215.08103}
}

@Article{kloosterman_fermat,
	Author = {Kloosterman, Remke},
	Title = {The zeta function of monomial deformations of {Fermat} hypersurfaces},
	FJournal = {Algebra \& Number Theory},
	Journal = {Algebra Number Theory},
	ISSN = {1937-0652},
	Volume = {1},
	Number = {4},
	Pages = {421--450},
	Year = {2007},
	Language = {English},
	DOI = {10.2140/ant.2007.1.421},
	zbMATH = {5529314},
	Zbl = {1166.14016}
}

@book {varma_local_global,
	AUTHOR = {Varma, Ila Kapur},
	TITLE = {On local-global compatibility for cuspidal regular algebraic
	automorphic representations of {GL}n},
	NOTE = {Thesis (Ph.D.)--Princeton University},
	PUBLISHER = {ProQuest LLC, Ann Arbor, MI},
	YEAR = {2015},
	PAGES = {72},
	ISBN = {978-1339-15725-2},
	MRCLASS = {Thesis},
	MRNUMBER = {3450147},
}

@misc{yang_monodromy,
	doi = {10.48550/ARXIV.2111.00318},
	
	url = {https://arxiv.org/abs/2111.00318},
	
	author = {Yang, Yuji},
	
	title = {An Ordinary Rank-Two Case of Local-Global Compatibility for Automorphic Representations of Arbitrary Weight Over CM Fields},
	
	publisher = {arXiv},
	
	year = {2021},
	
	copyright = {arXiv.org perpetual, non-exclusive license}
}

@Article{weil_jacobi_sum,
	Author = {Weil, Andr{\'e}},
	Title = {Jacobi sums as ``{Gr{\"o}{{\ss}}encharaktere}''},
	FJournal = {Transactions of the American Mathematical Society},
	Journal = {Trans. Am. Math. Soc.},
	ISSN = {0002-9947},
	Volume = {73},
	Pages = {487--495},
	Year = {1952},
	Language = {English},
	DOI = {10.2307/1990804},
	zbMATH = {3075867},
	Zbl = {0048.27001}
}

@Article{deligne_weil2,
	Author = {Deligne, Pierre},
	Title = {La conjecture de {Weil}. {II}},
	FJournal = {Publications Math{\'e}matiques},
	Journal = {Publ. Math., Inst. Hautes {\'E}tud. Sci.},
	ISSN = {0073-8301},
	Volume = {52},
	Pages = {137--252},
	Year = {1980},
	Language = {French},
	DOI = {10.1007/BF02684780},
	zbMATH = {3713855},
	Zbl = {0456.14014}
}

@InCollection{vogan93,
	Author = {Vogan, David A. jun.},
	Title = {The local {Langlands} conjecture},
	BookTitle = {Representation theory of groups and algebras},
	ISBN = {0-8218-5168-3},
	Pages = {305--379},
	Year = {1993},
	Publisher = {Providence, RI: American Mathematical Society},
	Language = {English},
	zbMATH = {432944},
	Zbl = {0802.22005}
}

@Article{faltings88,
	Author = {Faltings, Gerd},
	Title = {{{\(p\)}}-adic {Hodge} theory},
	FJournal = {Journal of the American Mathematical Society},
	Journal = {J. Am. Math. Soc.},
	ISSN = {0894-0347},
	Volume = {1},
	Number = {1},
	Pages = {255--299},
	Year = {1988},
	Language = {English},
	DOI = {10.2307/1990970},
	zbMATH = {2230},
	Zbl = {0764.14012}
}

@Article{roberts_villegas,
	Author = {Roberts, David P. and Rodriguez Villegas, Fernando},
	Title = {Hypergeometric motives},
	FJournal = {Notices of the American Mathematical Society},
	Journal = {Notices Am. Math. Soc.},
	ISSN = {0002-9920},
	Volume = {69},
	Number = {6},
	Pages = {914--929},
	Year = {2022},
	Language = {English},
	DOI = {10.1090/noti2491},
	zbMATH = {7579452},
	Zbl = {1506.11097}
}

@Article{shin12,
	Author = {Shin, Sug Woo},
	Title = {Automorphic {Plancherel} density theorem},
	FJournal = {Israel Journal of Mathematics},
	Journal = {Isr. J. Math.},
	ISSN = {0021-2172},
	Volume = {192},
	Pages = {83--120},
	Year = {2012},
	Language = {English},
	DOI = {10.1007/s11856-012-0018-z},
	zbMATH = {6127518},
	Zbl = {1300.22006}
}

@Article{erbach79,
	Author = {Erbach, D. W. and Fischer, J. and McKay, J.},
	Title = {Polynomials with {PSL}(2,7) as {Galois} group},
	FJournal = {Journal of Number Theory},
	Journal = {J. Number Theory},
	ISSN = {0022-314X},
	Volume = {11},
	Pages = {69--75},
	Year = {1979},
	Language = {English},
	DOI = {10.1016/0022-314X(79)90020-9},
	zbMATH = {3629071},
	Zbl = {0405.12011}
}

@Article{casselman80,
	Author = {Casselman, William},
	Title = {The unramified principal series of p-adic groups. {I}: {The} spherical function},
	FJournal = {Compositio Mathematica},
	Journal = {Compos. Math.},
	ISSN = {0010-437X},
	Volume = {40},
	Pages = {387--406},
	Year = {1980},
	Language = {English},
	zbMATH = {3739834},
	Zbl = {0472.22004}
}

@Misc{rodier82,
	Author = {Rodier, Fran{\c{c}}ois},
	Title = {Repr{\'e}sentations de {{\(\mathrm{GL}(n,k)\)}} o{\`u} {{\(k\)}} est un corps {{\(p\)}}-adique},
	Year = {1982},
	Language = {French},
	HowPublished = {S{\'e}min. {Bourbaki}, 34e ann{\'e}e, {Vol}. 1981/82, {Exp}. {No}. 587, {Asterisque} 92-93, 201-218 (1982).},
	zbMATH = {3798076},
	Zbl = {0506.22019}
}

@Article{zelevinsky80,
	Author = {Zelevinsky, A. V.},
	Title = {Induced representations of reductive {{\(p\)}}-adic groups. II: On irreducible representations of {{\(GL(n)\)}}},
	FJournal = {Annales Scientifiques de l'{\'E}cole Normale Sup{\'e}rieure. Quatri{\`e}me S{\'e}rie},
	Journal = {Ann. Sci. {\'E}c. Norm. Sup{\'e}r. (4)},
	ISSN = {0012-9593},
	Volume = {13},
	Pages = {165--210},
	Year = {1980},
	Language = {English},
	DOI = {10.24033/asens.1379},
	zbMATH = {3687666},
	Zbl = {0441.22014}
}

@Article{katz_messing,
	Author = {Katz, Nicholas M. and Messing, William},
	Title = {Some consequences of the {Riemann} hypothesis for varieties over finite fields},
	FJournal = {Inventiones Mathematicae},
	Journal = {Invent. Math.},
	ISSN = {0020-9910},
	Volume = {23},
	Pages = {73--77},
	Year = {1974},
	Language = {English},
	DOI = {10.1007/BF01405203},
	zbMATH = {3432442},
	Zbl = {0275.14011}
}

@Article{geraghty19,
	Author = {Geraghty, David},
	Title = {Modularity lifting theorems for ordinary {Galois} representations},
	FJournal = {Mathematische Annalen},
	Journal = {Math. Ann.},
	ISSN = {0025-5831},
	Volume = {373},
	Number = {3-4},
	Pages = {1341--1427},
	Year = {2019},
	Language = {English},
	DOI = {10.1007/s00208-018-1742-4},
	zbMATH = {7062604}
}

@Article{BLGHT,
	Author = {Barnet-Lamb, Tom and Geraghty, David and Harris, Michael and Taylor, Richard},
	Title = {A family of {Calabi}-{Yau} varieties and potential automorphy. {II}.},
	FJournal = {Publications of the Research Institute for Mathematical Sciences, Kyoto University},
	Journal = {Publ. Res. Inst. Math. Sci.},
	ISSN = {0034-5318},
	Volume = {47},
	Number = {1},
	Pages = {29--98},
	Year = {2011},
	Language = {English},
	DOI = {10.2977/PRIMS/31},
	zbMATH = {5883091},
	Zbl = {1264.11044}
}

@book {harris_taylor,
	AUTHOR = {Harris, Michael and Taylor, Richard},
	TITLE = {The geometry and cohomology of some simple {S}himura
	varieties},
	SERIES = {Annals of Mathematics Studies},
	VOLUME = {151},
	NOTE = {With an appendix by Vladimir G. Berkovich},
	PUBLISHER = {Princeton University Press, Princeton, NJ},
	YEAR = {2001},
	PAGES = {viii+276},
	ISBN = {0-691-09090-4},
	MRCLASS = {11G18 (11F70 11S37 14G35 22E45)},
	MRNUMBER = {1876802},
	MRREVIEWER = {James Milne},
}

@article {allen_newton,
AUTHOR = {Allen, Patrick B. and Newton, James},
TITLE = {Monodromy for some rank two {G}alois representations over {CM}
fields},
JOURNAL = {Doc. Math.},
FJOURNAL = {Documenta Mathematica},
VOLUME = {25},
YEAR = {2020},
PAGES = {2487--2506},
ISSN = {1431-0635},
MRCLASS = {11F80 (11R39)},
MRNUMBER = {4213130},
MRREVIEWER = {Andrea Conti},
}

@Misc{tate_background,
	Author = {Tate, J.},
	Title = {Number theoretic background},
	Year = {1979},
	Language = {English},
	HowPublished = {Automorphic forms, representations and {L}-functions, {Proc}. {Symp}. {Pure} {Math}. {Am}. {Math}. {Soc}., {Corvallis}/{Oregon} 1977, {Proc}. {Symp}. {Pure} {Math}. 33, 2, 3-26 (1979).},
	zbMATH = {3656073},
	Zbl = {0422.12007}
}

@Article{taylor_ladic_part2,
	Author = {Taylor, Richard},
	Title = {Automorphy for some {{\(l\)}}-adic lifts of automorphic mod {{\(l\)}} {Galois} representations. {II}},
	FJournal = {Publications Math{\'e}matiques},
	Journal = {Publ. Math., Inst. Hautes {\'E}tud. Sci.},
	ISSN = {0073-8301},
	Volume = {108},
	Pages = {183--239},
	Year = {2008},
	Language = {English},
	DOI = {10.1007/s10240-008-0015-2},
	zbMATH = {5528894},
	Zbl = {1169.11021}
}

@article{acampo_rigidity,
	author = {A'Campo, Lambert},
	title = "{Rigidity of Automorphic Galois Representations Over CM Fields}",
	journal = {International Mathematics Research Notices},
	volume = {2024},
	number = {6},
	pages = {4541-4623},
	year = {2023},
	month = {05},
	issn = {1073-7928},
	doi = {10.1093/imrn/rnad087},
	url = {https://doi.org/10.1093/imrn/rnad087},
	eprint = {https://academic.oup.com/imrn/article-pdf/2024/6/4541/56930640/rnad087.pdf},
}

@article{qian_potential,
	title={Potential Automorphy for {$GL_n$}}, 
	author={Qian, Lie},
	JOURNAL = {Invent. Math.},
	FJOURNAL = {Inventiones Mathematicae},
		YEAR = {2022},
	DOI = {https://doi.org/10.1007/s00222-022-01161-6},
}

@article {harris_lan_taylor_thorne,
	AUTHOR = {Harris, Michael and Lan, Kai-Wen and Taylor, Richard and
	Thorne, Jack},
	TITLE = {On the rigid cohomology of certain {S}himura varieties},
	JOURNAL = {Res. Math. Sci.},
	FJOURNAL = {Research in the Mathematical Sciences},
	VOLUME = {3},
	YEAR = {2016},
	PAGES = {Paper No. 37, 308},
	ISSN = {2522-0144},
	MRCLASS = {11G18 (11F75 14G35)},
	MRNUMBER = {3565594},
	MRREVIEWER = {Fausto Jarqu\'{\i}n Z\'{a}rate},
	DOI = {10.1186/s40687-016-0078-5},
	URL = {https://doi.org/10.1186/s40687-016-0078-5},
}

@Article{10author,
	Author = {Allen, Patrick B. and Calegari, Frank and Caraiani, Ana and Gee, Toby and Helm, David and Le Hung, Bao and Newton, James and Scholze, Peter and Taylor, Richard and Thorne, Jack A.},
	Title = {Potential automorphy over {CM} fields},
	FJournal = {Annals of Mathematics. Second Series},
	Journal = {Ann. Math. (2)},
	ISSN = {0003-486X},
	Volume = {197},
	Number = {3},
	Pages = {897--1113},
	Year = {2023},
	Language = {English},
	DOI = {10.4007/annals.2023.197.3.2},
	zbMATH = {7668527}
}

@article {calegari_geraghty_beyond,
	AUTHOR = {Calegari, Frank and Geraghty, David},
	TITLE = {Modularity lifting beyond the {T}aylor-{W}iles method},
	JOURNAL = {Invent. Math.},
	FJOURNAL = {Inventiones Mathematicae},
	VOLUME = {211},
	YEAR = {2018},
	NUMBER = {1},
	PAGES = {297--433},
	ISSN = {0020-9910},
	MRCLASS = {11F33 (11F80)},
	MRNUMBER = {3742760},
	MRREVIEWER = {Giovanni Rosso},
	DOI = {10.1007/s00222-017-0749-x},
	URL = {https://doi.org/10.1007/s00222-017-0749-x},
}

@article {allen_polarizable,
	AUTHOR = {Allen, Patrick B.},
	TITLE = {Deformations of polarized automorphic {G}alois representations
	and adjoint {S}elmer groups},
	JOURNAL = {Duke Math. J.},
	FJOURNAL = {Duke Mathematical Journal},
	VOLUME = {165},
	YEAR = {2016},
	NUMBER = {13},
	PAGES = {2407--2460},
	ISSN = {0012-7094},
	MRCLASS = {11F80 (11F70 11R34)},
	MRNUMBER = {3546966},
	MRREVIEWER = {Ivan Mati\'{c}},
	DOI = {10.1215/00127094-3477342},
	URL = {https://doi.org/10.1215/00127094-3477342},
}
\end{document}